\documentclass[11pt]{amsart}

\pdfoutput=1

\usepackage[text={420pt,660pt},centering]{geometry}

\usepackage{cancel}
\usepackage{esint,amssymb} 
\usepackage{graphicx}
\usepackage{MnSymbol}
\usepackage{mathtools} 
\usepackage[colorlinks=true, pdfstartview=FitV, linkcolor=blue, citecolor=blue, urlcolor=blue,pagebackref=false]{hyperref} \usepackage{microtype}

\usepackage{bm}
\usepackage{dsfont}
\usepackage{mathrsfs}
\usepackage{xcolor}
\usepackage{enumitem}

\definecolor{darkgreen}{rgb}{0,0.5,0}
\definecolor{darkblue}{rgb}{0,0,0.7}
\definecolor{darkred}{rgb}{0.9,0.1,0.1}

\newtheorem{proposition}{Proposition}
\newtheorem{theorem}[proposition]{Theorem}
\newtheorem{lemma}[proposition]{Lemma}

\newtheorem{result}[proposition]{Result}
\theoremstyle{remark}

\newtheorem{remark}[proposition]{Remark}

\theoremstyle{definition}
\newtheorem{definition}[proposition]{Definition}

\newtheorem*{theorem*}{Theorem}

\numberwithin{equation}{section}
\numberwithin{proposition}{section}
\numberwithin{figure}{section}
\numberwithin{table}{section}

\newcommand{\N}{\mathbb{N}}

\newcommand{\R}{\mathbb{R}}

\newcommand{\E}{\mathbb{E}}

\newcommand{\eps}{\varepsilon}

\renewcommand{\leq}{\leqslant}
\renewcommand{\geq}{\geqslant}

\renewcommand{\subset}{\subseteq}
\renewcommand{\bar}{\overline}
\renewcommand{\tilde}{\widetilde}

\newcommand{\Ll}{\left}
\newcommand{\Rr}{\right}
\renewcommand{\d}{\mathrm{d}}

\renewcommand{\fint}{\strokedint}

\newcommand{\mcl}{\mathcal}

\DeclareMathOperator{\tr}{tr}

\newcommand{\la}{\left\langle}
\newcommand{\ra}{\right\rangle}

\renewcommand{\H}{\mathsf{H}}
\newcommand{\pj}{\mathsf{p}}
\newcommand{\lf}{\mathsf{l}}
\renewcommand{\j}{{(j)}}
\renewcommand{\S}{S}

\renewcommand{\div}{\mathsf{div}}

\newcommand{\cH}{{L^2}}
\newcommand{\C}{\mathcal{Q}}
\newcommand{\D}{{D}}

\newcommand{\cL}{{\mathcal{L}}}
\newcommand{\sF}{{\mathsf{F}}}
\newcommand{\sP}{\mathscr{P}}
\newcommand{\HJ}{\mathrm{HJ}}
\newcommand{\bxi}{\boldsymbol{\xi}}

\newcommand{\bnabla}{\boldsymbol{\nabla}}
\newcommand{\supp}{\mathsf{supp}\,}
\newcommand{\seq}{\mathsf{seq}}
\newcommand{\itr}{\mathsf{int}}
\newcommand{\bsigma}{\boldsymbol{\sigma}}

\newcommand{\rv}[1]{}

\begin{document}

\author[Hong-Bin Chen]{Hong-Bin Chen}
\address[Hong-Bin Chen]{Institut des Hautes \'Etudes Scientifiques, France}
\email{hbchen@ihes.fr}


\title{Envelope representation of Hamilton--Jacobi equations from spin glasses}

\begin{abstract}
Recently, \cite{chen2023free} demonstrated that, if it exists, the limit free energy of possibly non-convex spin glass models must be determined by a characteristic of the associated infinite-dimensional non-convex Hamilton--Jacobi equation. In this work, we investigate a similar theme purely from the perspective of PDEs. Specifically, we study the unique viscosity solution of the aforementioned equation and derive an envelope-type representation formula for the solution, in the form proposed by Evans in \cite{evans2014envelopes}. The value of the solution is expressed as an average of the values along characteristic lines, weighted by a non-explicit probability measure. The technical challenges arise not only from the infinite dimensionality but also from the fact that the equation is defined on a closed convex cone with an empty interior, rather than on the entire space. In the introduction, we provide a description of the motivation from spin glass theory and present the corresponding results for comparison with the PDE results.
\end{abstract}

\maketitle

\section{Introduction}

We first review \rv{some} background on spin glass and describe the Hamilton--Jacobi equation approach in Sections~\ref{s.background_spin_glass} and~\ref{s.HJ_eqn_approach}. We emphasize that they only serve as the motivation for this work which focuses on the PDE aspects of the equation. In particular, no results from \rv{spin glass theory} are used. 
Then, in Section~\ref{s.envelop}, we state informally our main results on the envelope representation of the viscosity solution of the equation relevant in the spin glass setting. Lastly in Sections~\ref{s.rel_results} and~\ref{s.outline}, we briefly describe related results and give an outline of this work.

\subsection{Background on spin glass}\label{s.background_spin_glass}
A spin glass is a model of disordered magnetic materials in statistical physics. We describe a general vector spin glass model and then briefly review some history. Let $\D \in \N$ and let $P$ be a finite positive Borel measure on $\R^\D$ with a bounded support. We interpret $P$ as the distribution of each single spin, which takes value in $\R^\D$. For each $N\in\N$ denoting the size of the system (or the number of spins), we denote by $\sigma =(\sigma_{\bullet 1},\dots,\sigma_{\bullet N})\in (\R^\D)^N$ the spin configuration, where every single spin $\sigma_{\bullet i} = (\sigma_{di})_{1\leq d\leq \D}$ is viewed as a column vector. In this way, we also view $\sigma$ as a matrix in $\R^{\D\times N}$. For each spin configuration $\sigma$, we associate a real-valued quantity $H_N(\sigma)$ called the Hamiltonian, which is random in the spin glass model. A good way to define this quantity in general is to define it as a centered Gaussian field $(H_N(\sigma))_{\sigma\in\R^{\D\times N}}$ indexed by all spin configurations. This field is uniquely determined by its covariance: for any two spin configurations $\sigma,\sigma'$,
\begin{align}\label{e.H_N(sigma)}
    \E \Ll[H_N(\sigma)H_N(\sigma')\Rr]= N\xi \Ll(\sigma\sigma'^\intercal/N\Rr)
\end{align}
where $\xi:\R^{\D\times\D}\to\R$ is a fixed smooth function independent of $N$ and $\sigma'^\intercal$ is the matrix transpose of $\sigma'$.

For each $\beta\geq 0$ interpreted as the inverse temperature, we consider the free energy at size $N$:
\begin{align*}
    \mathsf{F}_N(\beta) =-\frac{1}{N} \E \log \int \exp\Ll(\beta H_N(\sigma)\Rr) \otimes_{i=1}^N \d P(\sigma_{\bullet i})
\end{align*}
where $\E$ averages over the randomness in $H_N(\sigma)$. We point out that the minus sign on the right-hand side is often omitted in mathematical literature on spin glass. The goal is to understand the limit of $\mathsf{F}_N(\beta)$ as $N$ tends to infinity for each fixed $\beta\geq 0$.

The most classical model is the Sherrington--Kirkpatrick (SK) model~\cite{sherrington1975solvable} where $\D=1$, $P= \delta_{-1}+\delta_{+1}$, and $\xi = |\cdot|^2$. In this case, $H_N(\sigma)$ can be constructed explicitly as
\begin{align*}
    H_N(\sigma) =\frac{1}{\sqrt{N}} \sum_{i,j=1}^N g_{ij} \sigma_i\sigma_j
\end{align*}
where $\sigma_i = \sigma_{\bullet i}$ is real-valued as $\D=1$ and $(g_{ij})$ is a collection of independent standard Gaussian random variables. In the sense of identifying the limit free energy, the SK model is considered solved. Parisi \cite{parisi1979infinite,parisi1980order,parisi1980sequence} (also see~\cite{parisi2021lecture}) first proposed the correct infinite-dimensional variational formula for the limit, which is now called the Parisi formula:
\begin{align}\label{e.parisi_sF}
    \lim_{N\to\infty} \mathsf{F}_N(\beta) = \sup_{\mu}\sP_\beta(\mu)
\end{align}
where $\sup$ is taken over probability measures $\mu$ on $[0,1]$ and $\sP_\beta$ is the so-called Parisi functional (usually the Parisi formula is written as an $\inf$; it is a $\sup$ here due to the extra minus sign in the definition of $\mathsf{F}_N(\beta)$).
On the rigorous side, Guerra and Toninelli~\cite{guerra2002thermodynamic} showed that $\lim_{N\to\infty}\mathsf{F}_N(\beta)$ exists, then Guerra \cite{gue03} gave the lower bound of $\lim_{N\to\infty}\mathsf{F}_N(\beta)$ by the Parisi formula, and later Talagrand~\cite{Tpaper} gave the matching upper bound. Panchenko \cite{pan} gave another proof of the upper bound using deeper results \cite{pan.aom} about the structure of the asymptotic Gibbs measure. By this approach, Panchenko was able to identify the limit free energy with a generalized Parisi formula in more general settings~\cite{pan.multi,pan.potts,pan.vec} as long as the function $\xi$ is convex.
We refer to~\cite{mourrat2024informal} for a short informal introduction to the Parisi formula.

However, there are interesting and important models with non-convex $\xi$. One of the simplest of such is the bipartite spin glass with $\D=2$, $P = (\delta_{-1}+\delta_{+1})^{\otimes 2}$, and $\xi(a) = a_{12}a_{21}$ for every $a=(a_{ij})_{1\leq i,j\leq 2}$. We can construct $H_N(\sigma)$ as
\begin{align*}
    H_N(\sigma) = \frac{1}{\sqrt{N}}\sum_{i,j=1}^N g_{ij}\sigma_{1,i}\sigma_{2,j}
\end{align*}
where $(g_{ij})$ still consists of independent standard Gaussians. We can think of the Hamiltonian in the bipartite model as arranging $\pm1$-valued spins into two layers (the two row-vectors in $\sigma$) and only counting the interaction between spins on different layers. This is in contrast to the SK model where every spin interacts with every other spin. Here, the function $\xi$ is not a convex function and the successful theory built around the Parisi formula is not applicable. 
It is explained in~\cite[Section~6]{mourrat2021nonconvex} that the limit free energy cannot be represented by a saddle-point variational formula or by arrangements of $\inf$'s and $\sup$'s. 

In general, it is not even known whether the limit of $\mathsf{F}_N(\beta)$ exists since Guerra and Toninelli's subadditive argument~\cite{guerra2002thermodynamic} breaks down when $\xi$ is not convex. In general, there is no prediction for the limit in terms of a variational formula. Hence, a new perspective is needed.

\subsection{Hamilton--Jacobi equation approach}\label{s.HJ_eqn_approach}

Throughout, for any two matrices or vectors $a$ and $b$ of the same dimension, we write $a\cdot b = \sum_{ij}a_{ij}b_{ij}$ as the entry-wise inner product and $|a|=\sqrt{a\cdot a}$. We denote by $\S^\D$ the linear space of $\D\times\D$ real symmetric matrices and $\S^\D_+\subset \S^\D$ the set of symmetric positive semi-definite matrices. For any matrix $a$ in $\S^\D_+$, we denote by $\sqrt{a}$ its matrix square root.

Let $H_N(\sigma)$ be given as in~\eqref{e.H_N(sigma)}.
For $N\in\N$, $t\geq0$, and $h\in\S^\D_+$, we consider the following Hamiltonian:
\begin{align}\label{e.H^t,h_N(sigma)=}
    H^{t,h}_N(\sigma) = \sqrt{2t}H_N(\sigma) - t N\xi\Ll(\sigma\sigma^\intercal/N\Rr)+ \sum_{i=1}^N \Ll(\sqrt{2h} z_i\Rr)\cdot \sigma_{\bullet i} - h\cdot \Ll(\sigma\sigma^\intercal\Rr)
\end{align}
where $(z_i)$ is a collection of independent standard Gaussian vectors in $\R^\D$.
We think of $\sqrt{2t}$ as a reparametrization of $\beta$ and $\sum_{i=1}^N \Ll(\sqrt{2h} z_i\Rr)\cdot \sigma_{\bullet i}$ as a Gaussian external field with intensity $h$ added to the system. The terms $t N\xi\Ll(\sigma\sigma^\intercal/N\Rr)$ and $h\cdot \Ll(\sigma\sigma^\intercal\Rr)$ are exactly one half of the variance of $\sqrt{2t}H_N(\sigma)$ and $\sum_{i=1}^N \Ll(\sqrt{2h} z_i\Rr)\cdot \sigma_{\bullet i}$, respectively. These two terms help us to \rv{form} the Hamilton--Jacobi (HJ) equation, to be introduced, when the \textit{self-overlap} $\sigma\sigma^\intercal/N$ is not constant.
We view them as correction terms similar to that appearing in an exponential martingale.

Instead of $\mathsf{F}_N(\beta)$, we consider the free energy
\begin{align}\label{e.F_N(t,h)}
    F_N(t,h) = -\frac{1}{N}\E\log \int \exp\Ll(H^{t,h}_N(\sigma)\Rr) \otimes_{i=1}^NP(\d \sigma_{\bullet i})
\end{align}
where $\E$ averages over the randomness in $H_N(\sigma)$ and $(z_i)$.
Notice that if $\sigma_{\bullet i}\sigma_{\bullet i}^\intercal=a$ for some $a\in\S^\D_+$ almost surely under $P$, then we have $F_N(\beta^2/2, 0) = \mathsf{F}_N(\beta) + \beta^2\xi(a)/2$. In general, we expect to relate $F_N(\beta^2/2, 0)$ to $\mathsf{F}_N(\beta)$ in a simple relation if $F_N$ converges and satisfies some regularity. 
In particular, this can be done when $\xi$ is convex.

We also consider the associated Gibbs measure
\begin{align*}
    \la\cdot\ra_{N,t,h} \propto \exp\Ll(H^{t,h}_N(\sigma)\Rr) \otimes_{i=1}^NP(\d \sigma_{\bullet i})
\end{align*}
where the left- and right-hand sides differ by a normalizing factor making $\la\cdot\ra_{N,t,h}$ a probability measure. Notice that it is a random measure since it depends on the randomness of $H^{t,h}_N(\sigma)$. We denote the tensorized version of $\la\cdot\ra_{N,t,h}$ still by itself. In this notation, we can sample two independent random variables $\sigma$ and $\sigma'$ from $\la\cdot\ra_{N,t,h}$.
A simple computation using Gaussian integration by parts yields
\begin{align}\label{e.dtF_N=...dhF_N=...}
    \partial_t F_N(t,h) = \E \la \xi \Ll(\sigma\sigma'^\intercal/N\Rr)\ra_{N,t,h},\qquad \partial_h F_N(t,h) = \E \la \sigma\sigma'^\intercal/N\ra_{N,t,h}.
\end{align}
Recall that $P$ is assumed to have a bounded support. So, $\Ll|\sigma\sigma'^\intercal/N\Rr|$ is always bounded by some constant $C_0$. Since $\xi$ is smooth, we can set $C_1$ to be the Lipschitz coefficient of $\xi$ restricted to the centered ball with radius $C_0$. Then, we can obtain from the above display that
\begin{align}\label{e.|dtF_N-xi(dhF_N)|<}
    \Ll|\partial_t F_N - \xi\Ll(\partial_h F_N\Rr)\Rr| \leq C_1\E\la  \Ll|\sigma\sigma'^\intercal/N - \E \la \sigma\sigma'^\intercal/N \ra\Rr|\ra
\end{align}
where we have omitted $(t,h)$ and subscripts $N,t,h$. In \rv{spin glass theory}, for small $t$, the right-hand side is expected to vanish as $N\to\infty$, namely, the \textit{overlap} $\sigma\sigma'^\intercal/N$ concentrates at its mean.
Therefore, at least for small $t$, we expect $F_N$ as a function of $(t,h)$ converges to the solution $f$ of the HJ equation
\begin{align*}
    \partial_t f - \xi (\partial_h f)=0,\qquad \text{on $[0,\infty)\times \S^\D_+$.}
\end{align*}
In our setting, since each spin is independently sampled from $P$, we have $F_N(0,\cdot) = F_1(0,\cdot)$. Hence, the initial condition for the equation is simply $f(0,\cdot) = F_1(0,\cdot)$.

However, for large $t$, this picture is not valid because $\sigma\sigma'^\intercal/N$ no longer concentrates and the system is in the so-called \textit{replica symmetry breaking} regime. In this case, the external field in~\eqref{e.H^t,h_N(sigma)=} fails to capture the asymptotics of $\sigma\sigma'^\intercal/N$ and we need to replace it with a more complicated external field parameterized by an integrable path from the following collection
\begin{align}\label{e.mclQ=}
    \mcl Q =\Ll\{q:[0,1)\to \S^\D_+ \ \big| \ q(s')-q(s)\in\S^\D_+,\, \forall s' \geq s;\ \text{$q$ is right-continuous}\Rr\} .
\end{align}
of increasing paths in $\S^\D_+$.
We keep the corresponding correction terms as mentioned below~\eqref{e.H^t,h_N(sigma)=}. This procedure gives a version of free energy that we denote as $F_N(t,q)$. 
We view $F_N$ as a function of $(t,q)$.
\rv{For the constant path $q = h\in\S^\D_+$, the expression reduces to $F_N(t,h)$ as defined in~\eqref{e.F_N(t,h)}.}
Similar to~\eqref{e.dtF_N=...dhF_N=...}, we can compute the derivatives of $F_N(t,q)$ in $t$ and in $q$. We only consider those $q$ that are square-integrable here, namely, paths belonging to
\begin{align}\label{e.mclQ_2=mclQcapL^2}
    \mcl Q_2  =\mcl Q \cap L^2\Ll([0,1),\S^\D\Rr).
\end{align}
Throughout, we denote the functional differentiation in $q$ as $\bnabla =\bnabla_q$ (whereas we denote the Euclidean gradient as $\nabla$). This can be defined with respect to the ambient $L^2$ space (see Definition~\ref{d.Frechet}) so that $\bnabla F_N(t,q)$ is viewed as a square-integrable path $\bnabla F_N(t,q,\cdot)$ from $[0,1)$ to $\S^\D$. 
For any measurable path $\kappa:[0,1)\to\R^{\D\times\D}$ that makes the following finite, we use the notation:
\begin{align}\label{e.bxi}
    \bxi (\kappa) = \int_0^1 \xi(\kappa(s))\d s.
\end{align}
A similar relation as in~\eqref{e.|dtF_N-xi(dhF_N)|<} holds, leading us to the following conjecture~\cite{mourrat2022parisi,mourrat2021nonconvex} that, in the case where $\xi$ is possibly non-convex, the free energy $F_N$ converges pointwise to the solution of the HJ equation
\begin{align}\label{e.hj}
    \partial_t f - \bxi(\bnabla f) = 0, \qquad \text{on $[0,\infty)\times \mcl Q_2$}
\end{align}
with initial condition $f(0,\cdot) = F_1(0,\cdot)$. Here, as explained previously, the nonlinearity at a point $(t,q)$ is understood as 
\begin{align*}
    \bxi(\bnabla f(t,q)) = \int_0^1 \xi\Ll(\bnabla f(t,q,u)\Rr)\d u.
\end{align*}
where the derivative $\bnabla f(t,q) = \bnabla f(t,q,\cdot)$ is defined as before for $\bnabla F_N(t,q)$.

When $\xi$ is convex, the limit of $F_N$ is still given by a Parisi-type variational formula similar to that in~\eqref{e.parisi_sF}, which is equivalent to the Hopf--Lax formula for~\eqref{e.hj} \cite{mourrat2022parisi,mourrat2020extending} (for a more general version, see~\cite[Theorem~1.1]{chen2023free}). It has been shown in~\cite{chen2022hamilton} that the Hopf--Lax formula is a representation for the unique Lipschitz viscosity solution of~\eqref{e.hj} when $\xi$ is convex. The well-posedness of this notion of solution in the general case (possibly non-convex $\xi$) relevant to the spin glass setting has also been proved in~\cite{chen2022hamilton}. In the non-convex case, Mourrat \cite{mourrat2021nonconvex,mourrat2023free} showed that the viscosity solution is always a lower bound:
\begin{align}\label{e.free_energy_lower_bd}
    \liminf_{N\to\infty} F_N(t,q)\geq f(t,q),\quad\forall (t,q)\in [0,\infty)\times \mcl Q_2.
\end{align}
This can be seen as a substitute for Guerra's lower bound, which does not work when $\xi$ is not convex. The matching upper bound is still missing. 

In the following, we describe the recent results for the non-convex model.
Recently, \cite{chen2023free} (and~\cite{chen2024free} for multi-species spin glasses) showed that any subsequential limit of $F_N$ is differentiable and satisfies~\eqref{e.hj} on a dense subset of $[0,\infty)\times \mcl Q_2$. However, this is insufficient to identify the limit of $F_N$, since \rv{uniqueness} is not guaranteed for solutions defined in a sense weaker than that of viscosity solutions. Also in~\cite{chen2023free}, a stronger property than the viscosity solution is shown: \rv{if the limit of $F_N$ exists, it} must be prescribed by a characteristic line. Let us explain this in more detail. 
In the discussion below, we write
\begin{align}\label{e.psi=F_1(0,cdot)}
    \psi= F_1(0,\cdot)
\end{align}
and \textit{assume that $F_N$ converges pointwise to some function $\bar f$ on $[0,\infty)\times \mcl Q_2$}. 
Under this assumption, \cite{chen2023free} showed that, for every $(t,q)\in[0,\infty)\times \mcl Q_2$, there is a point $q'\in\mcl Q_2$ such that
\begin{align}\label{e.q'}
    q = q' - t\nabla\xi (\bnabla \psi(q')).
\end{align}
Here, $\nabla\xi$ is the gradient of $\xi$ in the Euclidean space $\R^{\D\times\D}$. The derivative of the functional $p\mapsto \bxi(p) =\int_0^1\xi(p(s))\d s$ at any bounded $p:[0,1)\to\S^\D$ is given by $\bnabla \bxi(p) = \nabla\xi\circ p=\nabla\xi(p)$.\footnote{This relation reads that the functional derivative of $\bxi$ at $p$, viewed as an element in $L^2$, is equal to the path given $\nabla\xi\circ p : [0,1)\to \R^{\D\times\D}$, which we also write as $\nabla\xi(p)$ (the path $s\mapsto \nabla\xi(p(s))$).} Hence, one can interpret $s\mapsto q' - s \nabla\xi(\bnabla \psi(q'))$ as the characteristic line emitting from $q'$ associated with the equation~\eqref{e.hj}. We can restate~\eqref{e.q'} as that the characteristic line from $q'$ reaches $q$ at time $t$. Then, \cite{chen2023free} also showed 
\begin{align}\label{e.char_representation}
    \bar f(t,q) = \psi(q') + \la q-q', \bnabla\psi(q')\ra_{L^2} + t \bxi(\bnabla\psi(q'))
\end{align}
where the right-hand side is exactly the value associated with the characteristic. 
Moreover, for every $(t,q)$ in a dense subset of $[0,\infty)\times \mcl Q_2$, there is $q'$ satisfying~\eqref{e.q'}, \eqref{e.char_representation},
\begin{align}\label{e.dtf=xi(dqpsi),dqf=dapsi}
    \partial_t \bar f(t,q) = \bxi\Ll(\bnabla \psi(q')\Rr),\qquad\text{and}\qquad \bnabla \bar f(t,q)=\bnabla\psi(q').
\end{align}
In particular, the limit $\bar f$ is Gateaux differentiable at such $(t,q)$.

Recall that so far we are not able to identify the limit of free energy with the unique viscosity solution. On the one hand, any subsequential limit must satisfy the equation on a dense set, which is weaker than viscosity solutions; while the relation in~\eqref{e.char_representation} is a stronger property than that satisfied by the viscosity solution. 
(These could suggest that the notion of viscosity solutions might not be the most suitable one here.)
Hence, it is important to understand more about the viscosity solutions in this context, which motivates the endeavor in this work.

\subsection{Envelope representation of the viscosity solution}\label{s.envelop}

We want to study the representation of viscosity solutions in a form similar to the right-hand side in~\eqref{e.char_representation}. 
We emphasize that the analysis of the equation~\eqref{e.hj} in this work is purely based on the PDE toolbox. The spin glass background as described above only serves as the motivation and the guide for the acceptable assumptions imposed on~\eqref{e.hj}, for instance, the assumption on the initial condition (see~\ref{i.A}) and that on the nonlinearity (see~\ref{i.B}) in Section~\ref{s.r1}.

Our inspiration is drawn from the work by Evans~\cite{evans2014envelopes}. Let us explain the idea.
Notice that if we fix any $q'$ and let $(t,q)$ vary, then the expression on the right-hand side of~\eqref{e.char_representation} gives an affine solution
\begin{align}\label{e.envelope}
    (t,q) \mapsto \psi(q') + \la q-q', \bnabla\psi(q')\ra_{L^2} + t \bxi(\bnabla\psi(q'))
\end{align}
of the equation~\eqref{e.hj}. 
This is related to the two fundamental variational formulas, Hopf--Lax and Hopf. Replacing $\bnabla \psi(q')$ in~\eqref{e.envelope} by a free variable $p$, we consider the functional 
\begin{align*}
    \mcl J_{t,q}(q',p)= \psi(q') + \la q-q', p\ra_{L^2} + t \bxi(p)
\end{align*}
which is exactly the one considered in~\cite{chen2023free}. Then,
\begin{align*}
    (t,q)\quad\mapsto\quad \sup_{q'} \inf_p \mcl J_{t,q}(q',p)
\end{align*}
is the Hopf--Lax formula, which is the viscosity solution of~\eqref{e.hj} if $\xi$ is convex \cite{chen2022hamilton}. On the other hand,  
\begin{align*}
    (t,q)\quad\mapsto\quad  \sup_p \inf_{q'}\mcl J_{t,q}(q',p)
\end{align*}
is the Hopf formula, which is a viscosity solution if $\psi$ is convex \cite{chen2022hamilton}. (In the spin glass setting, $\psi$ is not convex in general.)
Both formulas can be seen as envelope constructions.
The extremality condition for $\mcl J_{t,q}(q',p)$ in $q'$ gives $p = \bnabla\psi(q')$, which recovers~\eqref{e.envelope} from $\mcl J_{t,q}(q',p)$.
Similar observations hold in finite dimensions, which is the setting for~\cite{evans2014envelopes}.

When neither the nonlinearity (here, $\xi$) nor the initial condition (here, $\psi$) is concave or convex, simple variational formulas are not available. In these cases, Evans proposed in~\cite{evans2014envelopes} a new representation in terms of a generalized envelope using affine solutions as in~\eqref{e.envelope}. Our main result is a generalization of Evans' result in the current setting, which is stated below in an informal way. 

\begin{result}[Envelope representation]\label{r.1}
Let $f$ be the unique Lipschitz viscosity solution of~\eqref{e.hj} with initial condition $f(0,\cdot)=\psi$ relevant in the spin glass setting.
For every $(t,q) \in \R\times \mcl Q_2$, there is a probability measure $\gamma_{t,{q}}$ on $\C_2$ such that
\begin{align}\label{e.r.1}
    f(t,{q}) = \int_{\C_2} \psi({q'})+\la {q}-{q'},\bnabla\psi({q'})\ra_\cH + t\bxi\Ll(\bnabla\psi({q'})\Rr)\ \d \gamma_{t,{q}}({q'}).
\end{align}
\end{result}

The rigorous version of it is Theorem~\ref{t}.
Notice that ~\eqref{e.r.1} is weaker than~\eqref{e.char_representation}, which is obtained through spin-glass techniques rather than purely PDE arguments used for Result~\ref{r.1}.

We mention a few technical difficulties.
First, $\mcl Q_2$ is a closed convex cone with an empty interior in $L^2$.
Hence, $\mcl Q_2$ is a domain with a boundary and its boundary is itself. Usually, PDEs on a domain with a boundary need to impose boundary conditions. Here, due to some monotonicity in the spin glass setting, there is no need for the boundary condition, which simplifies the analysis, especially in the infinite dimension.
Since we are adapting the results from~\cite{evans2014envelopes} which considered adjoint equations posed on the whole space, the boundary here makes the analysis of the adjoint equation harder. We need to do a $C^1$-extension of the initial condition and consider the corresponding equation on the whole space. The type of extension needed here seems to be new (see the discussion below Lemma~\ref{l.ext}).
Another difficulty is to ensure that $\gamma_{t,q}$ is supported on $\mcl Q_2$. For this, we need to obtain moment estimates and show the finite-dimensional approximations of $\gamma_{t,q}$ are asymptotically supported on the cone. The latter is needed also because we cannot control the $C^1$-extension outside the cone when the dimension increases.

Evans in~\cite{evans2014envelopes} also proved the representation of the derivatives at differentiable points. In our setting these results would correspond to that if $f$ is differentiable at a point $(t,q)$, then
\begin{align}\label{e.derivatives_rep}
    \partial_t f(t,{q})  = \int_{\C_2} \bxi\Ll(\bnabla\psi({q'})\Rr)\ \d \gamma_{t,{q}}({q'}),\qquad
    \bnabla f(t,{q})  = \int_{\C_2} \bnabla\psi({q'})\ \d \gamma_{t,{q}}({q'}).
\end{align}
In finite dimensions, these relations hold at almost every point by Rademacher's theorem.
To prove~\eqref{e.derivatives_rep} in infinite dimensions, there are some difficulties. First, we need a good notion of differentiability. It is tempting to work with \rv{Fr\'echet} differentiability since it gives control uniformly in all directions. But, then we do not know if $f$ is Fr\'echet differentiable ``almost everywhere'' given that $f$ is Lipschitz. There are results generalizing Rademacher's theorem to infinite dimensions, for instance, \cite{lindenstrauss2003frechet,asplund1968frechet,johnson2002almost,preiss1990differentiability}. But they concern functions defined on an open set, which is not applicable here because $\mcl Q_2$ has an empty interior in $L^2$. 
Another option is to work with the weaker notion of \rv{Gateaux} differentiability. Indeed, it is proven in~\cite[Proposition~2.7]{chen2023free} that any Lipschitz function on $\mcl Q_2$ or $[0,\infty)\times\mcl Q_2$ is Gateaux differentiable ``almost everywhere''. However, \rv{Gateaux} differentiability offers much less control over the derivative.
If we do not worry about the density or even existence of differentiable points and we restrict our attention to only showing~\eqref{e.derivatives_rep} at a Fr\'echet differentiable point $(t,q)$, then there would still be technical issues. Indeed, we want to approximate \rv{Fr\'echet} derivatives of $f$ at $(t,q)$ by derivatives of finite-dimensional equations at points $(t_j,x_j)$ converging to $(t,q)$ as $j$ tends to infinity. To ensure that $\gamma_{t,q}$ is supported on the cone $\mcl Q_2$ and to make the approximation work, we need to guarantee that $(t_j,x_j)$ lies in the interior of a finite-dimensional cone (projection of $\mcl Q_2$). This is again nontrivial since $\mcl Q_2$ has an empty interior and it is easy for $(t_j,x_j)$ to be away from the cone.

Here, we resort to a weaker statement. We give a representation of the form in~\eqref{e.derivatives_rep} but only for limits of derivatives of finite-dimensional approximations $(f_j)_{j\in\N}$ (see~\eqref{e.f_j=} and Lemma~\ref{l.f_j_approx}), which are solutions of finite-dimensional projections of the equation~\eqref{e.hj}. They are used in~\cite{mourrat2021nonconvex,mourrat2023free} to prove the lower bound on the free energy as in~\eqref{e.free_energy_lower_bd} and they are naturally associated with finite-step replica symmetry breakings in the spin glass phase.
Below is our second main result.

\begin{result}[Envelope representation for limits of derivatives] \label{r.2}
We consider projections of the HJ equation~\eqref{e.hj} (with $f(0,\cdot)=\psi$) into finite-dimensional HJ equations indexed by $j\in\N$. Let $(f_j)_{j\in\N}$ be the solution of the $j$-th equation and let $(t_j,x_j)$ be a differentiable point of $f_j$. Then, the sequence 
\begin{align}\label{e.(Df_j...)_intro}
    \Big(\,\partial_t f_j(t_j,x_j),\ \nabla f_j(t_j,x_j)\,\Big)_{j\in\N},
\end{align}
is precompact.
Assume that $(t_j, x_j)_{j\in\N}$ converges to some $(t, q)\in[0,\infty)\times \mcl Q_2$. Then, for any subsequential limit $(a,p)\in\R\times L^2$ of~\eqref{e.(Df_j...)_intro}, there is a Borel probability measure $\gamma_{t,q}$ on $\mcl Q_2$ such that
\begin{gather}\label{e.r.2(1)}
\begin{split}
    f(t,{q}) - \la {q}, p\ra_\cH -ta = \int_{\C_2} \psi({q'}) -\la{q'}, \bnabla\psi({q'})\ra_{\cH}  \d \gamma_{t,{q}}({q'}),
    \\
    a = \int_{\C_2} \bxi\left(\bnabla\psi({q'})\right)\d \gamma_{t,{q}}({q'}),\qquad p = \int_{\C_2} \bnabla\psi({q'}) \d \gamma_{t,{q}}({q'}).
\end{split}
\end{gather}
In particular, since we always have $a=\bxi(p)$, the measure $\gamma_{t,q}$ also satisfies
\begin{align}\label{e.r.2(2)}
    \int_{\C_2} \bxi\left(\bnabla\psi({q'})\right)\d \gamma_{t,{q}}({q'}) = \bxi\left(\int_{\C_2} \bnabla\psi({q'})\d \gamma_{t,{q}}({q'})\right).
\end{align}
\end{result}

Here, the convergence of $x_j$ to $q$ and that of $\nabla f_j(t_j,x_j)$ to $p$ are understood in $L^2$ when we lift them (see~\eqref{e.pj,lj=}) from finite-dimensions to $L^2$-valued objects.
Again, the representations in~\eqref{e.r.2(1)} are similar to but weaker than those in~\eqref{e.char_representation} and~\eqref{e.dtf=xi(dqpsi),dqf=dapsi}.
Result~\ref{r.2} is a generalization of Result~\ref{r.1}. Indeed, substituting the expressions for $a$ and $p$ from \rv{the second line of}~\eqref{e.r.2(1)} into the first line recovers~\eqref{e.r.1}; and, for each $(t,q)\in[0,\infty)\times \mcl Q_2$, we can always find differentiable points $(t_j,x_j)$ converging to $(t,q)$. Notice that the commutation of $\bxi$ and $\int \d \gamma_{t,q}$ in~\eqref{e.r.2(2)} is nontrivial and gives interesting restrictions on the class of measures which $\gamma_{t,q}$ belongs to.

Lastly, it is interesting to understand the uniqueness of $\gamma_{t,q}$ and the condition for $\gamma_{t,q}$ to be a Dirac measure in the general non-convex setting. These questions are not trivial even in finite dimensions. Hence, we postpone them to future investigations.

\subsection{Other related works}
\label{s.rel_results}
Besides the relevant works already mentioned, we describe some other related works.

The original Parisi formula has been extended to various settings: the SK model with soft spins \cite{pan05}, the SK model with spherical spins \cite{tal.sph}, the mixed $p$-spin model \cite{panchenko2014parisi,pan}, the multi-species model (Ising spins and spherical spins) \cite{barcon,pan.multi,bates2022free}, the Potts spin glass~\cite{pan.potts,bates2023parisi,chen2024parisi}, and general vector-spin models \cite{pan.vec}. 
A crucial ingredient to most of these is Panchenko's ultrametricity result~\cite{pan.aom}.
Recently, there are works showing that the Parisi formula can be inverted into an infimum instead of supremum~\cite{mourrat2023inverting,issa2024hopf} (in the current notation here; in other literature, it is often the other way around).
The HJ approach first appeared in Guerra's work \cite{guerra2001sum} in the replica-symmetric regime, which was later extended to various settings~\cite{barra1,barra2,abarra,barramulti,barra2014quantum} (also see~\cite{chen2022pde} relating the cavity computation to the HJ equation).

Tools from \rv{spin glass theory} have also been successfully applied to statistical inference models. In particular, in the optimal Bayesian inference, due to the presence of the so-called Nishimori identity, the system is always replica-symmetric. As a result, the variational formula for the limit free energy is finite-dimensional~\cite{barbier2016, barbier2019adaptive, barbier2017layered, chen2022statistical, HB1, HBJ, chen2023freemulti-layer, kadmon2018statistical,   lelarge2019fundamental, lesieur2017statistical, luneau2020high, mayya2019mutualIEEE, miolane2017fundamental, mourrat2020hamilton, mourrat2021hamilton, reeves2020information, reeves2019geometryIEEE}. Similar to the spin glass setting, the relevant HJ equation is defined on a convex cone but in finite dimensions, which is studied in~\cite{chen2022hamiltonCones,chen2020fenchel}. Related works on the well-posedness of these equations also include~\cite{issa2024weaksol,issa2024weak}.
In the non-optimal inference model, the system is no longer replica-symmetric and thus better resembles spin glasses~\cite{camilli2022inference,barbier2021performance,barbier2022price,pourkamali2022mismatched,guionnet2023estimating}.
Infinite-dimensional equations similar to~\eqref{e.hj} also arise in the setting of inference on sparse graphs~\cite{dominguez2022infinite,dominguez2022mutual,dominguez2024critical}.
The HJ technique is also useful in treating spin glass models with additional conventional order parameters such as the self-overlap or the mean magnetization~\cite{chen2024self,chen2024conventional}.
We refer to the book~\cite{dominguez2023statistical} for an overview of the HJ approach to statistical mechanics.

We refer to the classic work~\cite{crandall1992user} for HJ equations in finite dimensions. We focus more on equations in infinite dimensions. First-order HJ equations in Banach spaces with Radon-Nikodym property and, in particular, Hilbert spaces have been studied first in~\cite{crandall1985hamiltonI,crandall1986hamiltonII}. Other works on HJ equations in Hilbert spaces include \cite{crandall1986hamiltonIII,crandall1990viscosityIV,crandall1991viscosityV,crandall1994hamiltonVI,tataru1992viscosity,feng2006large,feng2009comparison}. Works on equations in Wasserstein spaces include \cite{fenkur,cardaliaguet2008deterministic,cardaliaguet2010notes,cardaliaguet2012differential,gangbo2008hamilton,amb14,gangbo2014optimal,gangbo2019differentiability,feng2012hamilton}.

The technical tool used by Evans in~\cite{evans} is the adjoint equation of the linearization of the original HJ equation, which has previously appeared in~\cite{evans2010adjoint,tran2011adjoint,cagnetti2011aubry}. This method will also be used here to analyze finite-dimensional approximations of the infinite-dimensional equation.

\subsection{Outline}\label{s.outline}

In Section~\ref{s.def_result}, we give detailed definitions and state the rigorous versions of Results~\ref{r.1} and~\ref{r.2} in Theorems~\ref{t} and~\ref{t2}, respectively. In particular, we define the notion of viscosity solutions suitable in both finite and infinite dimensions and we introduce the finite-dimensional approximations appearing in Result~\ref{r.2}. 

In Section~\ref{s.fin_d_PDE}, we study the properties of the solutions of finite-dimensional equations, including the aforementioned approximations and other related equations. Here, we use the ideas from~\cite{evans2014envelopes} applied to the finite-dimensional approximations. Instead of the final form of the result in~\cite{evans2014envelopes}, we use the form still with error terms since our procedure of taking limits is more complicated than and different from that in~\cite{evans2014envelopes}. Apart from this, we derive moment estimates on the measures whose limit is $\gamma_{t,q}$ and we show that these measures are asymptotically supported on the cone. 
We also include other analytic results such as the tightness of these measures and a $C^1$-extension lemma for functions defined on a cone.

Lastly, in Section~\ref{s.proofs}, we combine all the results to first prove Result~\ref{r.2} (Theorem~\ref{t2}) and use it to derive Result~\ref{r.1} (Theorem~\ref{t}).

\section{Definitions and main results}\label{s.def_result}

Recall $\mcl Q$ in~\eqref{e.mclQ=} and that we endow $\S^\D$ with the entry-wise inner product (see the first paragraph in Section~\ref{s.HJ_eqn_approach}).
Consistently as in~\eqref{e.mclQ_2=mclQcapL^2}, for $r\in[1,\infty]$, we write
\begin{align}\label{e.L^p}
    L^r = L^r([0,1),\S^\D),\qquad \mcl Q_r = \mcl Q \cap L^r.
\end{align}
We denote the norm in $L^r$ as $|\cdot|_{L^r}$ and the inner product in $L^2$ as $\la\cdot,\cdot\ra_{L^2}$.

\subsection{Differentiability and viscosity solutions}

\begin{definition}[Fr\'echet differentiability]\label{d.Frechet}

For a subset $G$ of a Hilbert space $E$ and ${q}\in G$,
a function $g:G\to\R$ is \textbf{Fr\'echet differentiable} at ${q}$ if there is a unique element $p\in E$ such that
\begin{align}\label{e.frech}
    \lim_{r\to 0}\sup_{\substack{{q'}\in G\setminus\{{q}\}\\ |{q'}-{q}|_E\leq r}}\frac{\Ll|g({q'}) - g({q}) - \la p , {q'}-{q}\ra_E\Rr|}{|{q'}-{q}|_E} =0.
\end{align}
In this case, we denote by $p=\bnabla g({q})$ the derivative of $g$ at $q$.
\end{definition}

In the above, $|\cdot|_E$ is the norm of $E$ and $\la\cdot,\cdot\ra_E$ is the inner product.
We can rewrite~\eqref{e.frech} as
\begin{align*}
    g({q'}) - g({q}) = \la {q'}-{q}, \bnabla g({q}) \ra_E + o\Ll(\Ll|{q'}-{q}\Rr|_E\Rr)
\end{align*}
for ${q'} \in G$ tending to ${q}$.
In this work, aside from obvious applications in Euclidean spaces, we take $E$ to be either $L^2$ or $\R\times L^2$ and $G$ to be either $\mcl Q_2$ or $[0,\infty)\times \mcl Q_2$, respectively. 
For these two choices, since the linear span of $G$ is dense in $E$, we conclude that if there is $p$ for~\eqref{e.frech} to hold, then $p$ must be unique.

We define the notion of smoothness.

\begin{definition}[Smoothness]\label{d.smooth}
For a subset $G$ of a Hilbert space $E$, a function 
$g:G \to\R$ is said to be \textbf{smooth} if $g$ is Fr\'echet differentiable everywhere, the function ${q}\mapsto \bnabla g({q})$ is continuous, and it holds at every ${q}$ that
    \begin{align*}
    \limsup_{r\to 0}\sup_{\substack{{q'}\in G\setminus\{{q}\}\\ |{q'}-{q}|_E\leq r}}\frac{\Ll|g({q'}) - g({q}) - \la {q'}-{q},\, \nabla g({q}) \ra_{E}\Rr|}{|{q'}-{q}|_{E}^2} <\infty.
\end{align*}   
\end{definition}

The displayed condition can be rewritten as 
\begin{align*}
    g({q'}) - g({q}) - \la {q'}-{q},\, \bnabla g({q}) \ra_{E} + O\Ll(|{q'}-{q}|_{E}^2\Rr)
\end{align*}
as ${q'}$ approaches ${q}$. 

We choose to denote the gradient operator as $\bnabla$ when $E$ is infinite-dimensional to emphasize the infinite dimensionality and that it is understood in the strong sense of Fr\'echet. In finite dimensions, we simply use $\nabla$ to denote the usual Euclidean gradient operator.

\begin{definition}[Viscosity solution]\label{d.vis_sol}
For a subset $ G$ of a Hilbert space $E$, and a function $\sF:E\to\R$, a \textbf{viscosity subsolution} (resp.\ \textbf{supersolution}) of
\begin{align}\label{e.dtg-F(dg)=0}
    \partial_t f - \sF(\nabla f) =0,\quad \text{on }[0,\infty)\times G,
\end{align}
is a continuous function $f:[0,\infty)\times G\to\R$ such that, for every $(t,{q})\in (0,\infty)\times G$ and every smooth $\phi:(0,\infty)\times G\to\R$ satisfying that $f-\phi$ has a local maximum (resp.\ minimum) at $(t,{q})$, it holds that $\Ll(\partial_t \phi - \sF(\nabla \phi)\Rr)(t,{q})\leq 0$ (resp.\ $\geq 0$). If $g$ is both a subsolution and supersolution, then $g$ is said to be a \textbf{viscosity solution} of~\eqref{e.dtg-F(dg)=0}.

We denote by $\HJ(G,\sF;f_0)$ the equation~\eqref{e.dtg-F(dg)=0} together with an initial condition $f(0,\cdot)=f_0$ on $G$ (for some $f_0$ defined on a domain possibly larger than $G$), where the ambient Hilbert space should be clear from the context.
\end{definition}

We also need the notion of dual cones in the general setting and the associated monotonicity.

\begin{definition}[Dual cone]\label{d.dual_cone}
For a nonempty closed convex cone $K$ in a Hilbert space $E$, its \textbf{dual cone} is $K^*=\{x\in E:\la x,y\ra_E\geq 0,\ \forall y\in K\}$. Any function $g:G\to \R$ defined on a subset $G$ of $K$ is said to be \textbf{$K^*$-increasing} provided $g(q')\geq g(q)$ whenever $q'-q \in K^*$.
\end{definition}

It is classical that $\S^\D_+$ is self-dual in $\S^\D$ in the sense that $(\S^D_+)^* =\S^\D_+$, which can be verified through diagonalization. We will recall an explicit expression of the dual cone $\mcl Q_2^*$ of $\mcl Q_2$ in Lemma~\ref{l.dual_cone}.

\subsection{Precise statement of Result~\ref{r.1}}\label{s.r1}

We start by rigorously describing the HJ equation relevant to the spin glass setting. 

First, we describe the initial condition.
Recall that in the introduction (see~\eqref{e.psi=F_1(0,cdot)}), we consider $\psi$ equal to $F_1(0,\cdot)$ which is the relevant choice in the spin glass setting. Since we are directly working on the PDE side, we consider a general initial condition $\psi$ that satisfies some properties of $F_1(0,\cdot)$. Henceforth, we assume the following:
\begin{enumerate}[label={\rm (A)}]
    \item \label{i.A} The function $\psi:\mcl Q_2\to\R$ is smooth (as in Definition~\ref{d.smooth}) and its derivative satisfies $\bnabla \psi(q) \in \mcl Q_\infty$ at every $q\in\mcl Q_2$, $\sup_{q\in\mcl Q_2}|\bnabla \psi(q)|_{L^\infty}\leq 1$, and the function $q\mapsto \bnabla \psi(q)$ is Lipschitz on $\mcl Q_2$ (in the $L^2$-metric).
\end{enumerate}
By~\cite[Corollary~5.2]{chen2023free}, these properties of $\psi$ indeed hold in the spin glass case. The number $1$ in $\sup_{q\in\mcl Q_2}|\bnabla \psi(q)|_{L^\infty}\leq 1$ comes from the usual assumption in vector spin glass that each single spin lies in the unit ball of $\R^\D$. In general, this assumption is not restrictive since re-scaling is always possible.

Then, we describe the nonlinearity. Recall that in the introduction, $\xi$ is a function characterizing the Gaussian disorder of the spin glass Hamiltonian as in~\eqref{e.H_N(sigma)}.
\rv{We assume the following}:
\begin{enumerate}[label={\rm (B)}]
    \item \label{i.B} The function $\xi:\R^{\D\times \D}\to\R$ is smooth (thus locally Lipschitz). Its derivative satisfies $\nabla \xi(a) \in\S^\D_+$ at every $a\in\S^\D_+$ and $\nabla \xi(a') -\nabla\xi (a)\in\S^\D_+$ whenever $a'-a \in \S^\D_+$ and $a\in\S^\D_+$.
\end{enumerate}
We can interpret the condition on $\nabla\xi$ as that both $\xi$ and $\nabla\xi$ are increasing on $\S^\D_+$. In the spin glass setting, if we assume that $\xi$ satisfies~\eqref{e.H_N(sigma)} for some Gaussian field $(H_N(\sigma))_{\sigma\in\R^{\D\times N}}$ and that $\xi$ admits an absolutely convergent power series expansion, then~\cite[Proposition~6.6]{mourrat2023free} states that $\xi$ must be of the form $\xi(a) = \sum_{p=0}^\infty \mathsf{C}^{(p)}\cdot a^{\otimes p}$ for every $a\in \R^{\D\times\D}$, where $\mathsf{C}^{(p)}$ is a matrix in $\S^{\D^p}_+$, $a^{\otimes p}$ is the $p$-th order tensor product of $a$ viewed as a matrix in $\R^{\D^p\times \D^p}$, and $\cdot$ is the entry-wise inner product. In this case, it is straightforward to verify~\ref{i.B} using the power series expansion and the fact that the coefficients are positive semi-definite matrices. Therefore, \ref{i.B} is a natural condition to put on $\xi$.

From $\xi$ to the nonlinearity in the infinite-dimensional equation, we need some modification.
Given any $\xi$ satisfying~\ref{i.B}, \cite[Lemma~4.3]{chen2022hamilton} gives the existence of a \textbf{regularization} $\bar \xi:\S^\D_+\to\R$ of $\xi$ which is defined to satisfy
\begin{enumerate}[start=1,label={\rm (c\arabic*)}]
    \item $\bar \xi$ coincides with $\xi$ on $\{a \in \S^\D_+: |a_{ij}|\leq 1,\,\forall i,j\}$; \label{i.c1}
    \item $\bar \xi$ is Lipschitz, bounded from below, $\S^\D_+$-increasing, and satisfies that for every $b\in\S^\D_+$, we have $\bar\xi (a'+b) - \bar\xi(a') \geq \bar\xi (a+b) - \bar\xi(a)$ whenever $a'-a\in\S^\D_+$.
    \item $\bar \xi$ is convex, if, in addition, $\xi$ is convex.
\end{enumerate}
Associated with a regularization $\bar \xi$, we define $\H:\cH\to \R$ by
\begin{align}\label{e.H=}
    \H(\kappa) = \inf \Ll\{\int_0^1 \bar \xi ({q}(s))\d s:\:{q} \in \C_2\cap \Ll(\kappa + \C_2^*\Rr)\Rr\},\quad\forall \kappa \in \cH.
\end{align}
By \cite[Lemma~4.5]{chen2022hamilton}, the following holds:
\begin{enumerate}[start=1,label={\rm (d\arabic*)}]
    \item $\H({q}) = \int_0^1 \bar \xi ({q}(s))\d s$ for every ${q}\in\C_2$; \label{i.d1}
    \item \label{i.d2} $\H$ is Lipschitz (in the $L^2$-metric), bounded from below, and $\mcl Q_2^*$-increasing (as in Definition~\ref{d.dual_cone}).
    \item $\H$ is convex if $\xi$ is convex.
\end{enumerate}

We are ready to state the precise version of Result~\ref{r.1}.

\begin{theorem}[Envelope representation]\label{t}
Let $\psi$ satisfy~\ref{i.A}. Let $\xi$ satisfy~\ref{i.B}, let $\bar \xi$ be a regularization of $\xi$, and let $\H$ be given as in~\eqref{e.H=}.
Let $f$ be the unique Lipschitz viscosity solution of $\HJ(\C_2,\H;\psi)$ (as in Definition~\ref{d.vis_sol}).
Then, for every $(t,{q})\in[0,\infty)\times \mcl Q_2 $ there is a Borel probability measure $\gamma_{t,{q}}$ on $\C_2$ such that
\begin{align}\label{e.t}
    f(t,{q}) = \int_{\C_2} \psi({q'})+\la {q}-{q'},\bnabla\psi({q'})\ra_\cH + t\bxi(\bnabla\psi({q'}))\d \gamma_{t,{q}}({q'}).
\end{align}
\end{theorem}

Recall $\bxi$ from~\eqref{e.bxi}. Due to the boundedness of $\bnabla\psi$ in~\ref{i.A}, the term $\bxi\left(\bnabla\psi\right)$ is indeed finite. As in Definition~\ref{d.vis_sol}, the equation $\HJ(\C_2,\H;\psi)$ is understood as
\begin{align*}
\begin{cases}
    \partial_t f - \H (\nabla f) = 0, \quad &\text{on }[0,\infty)\times \C_2,
    \\
    f(0,\cdot) = \psi, &\text{on }\C_2.
\end{cases}
\end{align*}
We view this as a modification of~\eqref{e.hj}.
We have two remarks to make.

\begin{remark}[Interpretation of~\eqref{e.hj}]
In the informal statement, $f$ is taken to be the ``unique viscosity solution'' of~\eqref{e.hj}, which is different from the modified equation $\HJ(\C_2,\H;\psi)$ in Theorem~\ref{t}. Let us justify this choice. First, since the functional derivative is viewed as an element in $L^2$ (see Definition~\ref{d.Frechet}), for general $\xi$ (for instance, $\xi(a) = |a|^3$) it is possible that $\bxi$ (defined in~\eqref{e.bxi}) is not finite on $L^2$. Now, if $\bar\xi$ is a regularization, then its Lipschitzness ensures that $\bar\xi$ grows at most linearly and thus $\int_0^1 \bar \xi(p(s))\d s$ is finite at every $p\in L^2$. Also, since $\bnabla F_N(t,q) \in \mcl Q_\infty$ and $|\bnabla F_N(t,q)|_{L^\infty}\leq 1$ at every $(t,q)$ in the spin glass setting (see~\cite[Proposition~5.1]{chen2023free}), heuristically we expect $f$ as the limit to satisfy $\bnabla f(t,q) \in \mcl Q_\infty$ and $|\bnabla f(t,q)|_{L^\infty}\leq 1$. Then, by \ref{i.c1} for $\bar\xi$ and~\ref{i.d1} for $\H$, heuristically, we always have $\bxi(\bnabla f) = \H(\bnabla f)$. In this way, we interpret the equation~\eqref{e.hj} as $\HJ(\C_2,\H;\psi)$, which is exactly the interpretation used in~\cite{chen2022hamilton} (see Definition~4.2 therein). 
\end{remark}

\begin{remark}[Well-posedness and independence of regularization]
It was shown in~\cite[Theorem~4.6]{chen2022hamilton} that, under the condition in Theorem~\ref{t}, there is a unique Lipschitz viscosity solution $f$ of $\HJ(\C_2,\H;\psi)$. Moreover, ~\cite[Theorem~4.6~(1)]{chen2022hamilton} shows that $f$ is always the limit of finite-dimensional approximations independent of $\bar \xi$, which implies that indeed our interpretation of~\eqref{e.hj} does not \rv{depend} on the choice of regularization.
\end{remark}

\subsection{Precise statement of Result~\ref{r.2}}

We start with describing a natural way to approximate the HJ equation~\eqref{e.hj}, which requires the introduction of projections and lifts.

Recall that, for $a,b\in \S^\D$, the inner product is written as $a\cdot b =\sum_{i,j=1}^\D a_{i,j}b_{i,j}$. For $j\in\N$, let $\cL^j$ be the linear space with the inner product given by
\begin{align}\label{e.<x,y>_cL^j=} 
    \cL^j = (\S^\D)^{2^j},\qquad \la x,y\ra_{\cL^j} = \frac{1}{2^j} \sum_{k=1}^{2^j}x_k\cdot y_k
\end{align}
where $x = (x_i)_{1\leq k\leq 2^j} \in \S^\D$ and similarly for $y$. Notice the normalization factor $2^{-j}$.

We associate with $j \in\N$ a dyadic partition $([\frac{k-1}{2^j},\frac{k}{2^j}))_{1\leq k\leq 2^j}$ of $[0,1)$. We define projection $\pj_j:\cH\to\cL^j$ and lift $\lf_j :\cL^j\to\cH$ by
\begin{gather}\label{e.pj,lj=}
    \pj_j \kappa = \left(2^j\int_{\frac{k-1}{2^j}}^\frac{k}{2^j}\kappa(s)\d s\right)_{1\leq k \leq 2^j},\quad\forall \kappa \in\cH; \qquad\qquad
    \lf_j a = \sum_{k=1}^{2^j}a_k \mathds{1}_{\big[\frac{k-1}{2^j}, \frac{k}{2^j}\big)},\quad\forall a \in \cL^j.
\end{gather}

For each $j\in\N$, let $\mcl Q^j_2$ be given by
\begin{align}\label{e.Q^j_2=}
    \mcl Q^j_2 =  \Ll\{x\in (\S^\D_+)^{2^j}\subset \cL^j:\quad  x_k-x_{k-1} \in \S^\D_+,\ \forall k\in\{2,\dots,2^j\}\Rr\}\subset \cL^j,
\end{align}
which is the closed convex cone whose elements contain positive semi-definite and increasing entries. 
In Lemma~\ref{l.proj_cones}, we will verify $\pj_j(\mcl Q_2)=\mcl Q^j_2$.
We denote by $\itr(\mcl Q^j_2)$ the interior of $\mcl Q^j_2$ in $\cL^j$.

For a function $g$ defined on a subset of $\R\times \cH$ (resp.\ $\cH$), we define its \textbf{$j$-projection} $g^j$ by
\begin{align}\label{e.j-projection}
    \text{$g^j(t,a) = g(t,\lf_j a)$ \qquad (resp.\ $g^j(a) = g(\lf_ja)$)}.
\end{align}
For a function $h$ defined on a subset of $\R\times \cL^j$ (resp.\ $\cL^j$), we define its \textbf{lift} $h^\uparrow$ by 
\begin{align}\label{e.lift=}
    \text{$h^\uparrow(t,{q}) = h(t,\pj_j{q})$ 
    \qquad (resp.\ $h^\uparrow({q}) = h(\pj_j{q})$).}
\end{align}

Recall the definition of $\bxi$ in~\eqref{e.bxi} and let $\psi$ be the initial condition given in~\ref{i.A}.
For each $j\in\N$, we consider their $j$-projections $\bxi^j$ and $\psi^j$ which have the following expressions:
\begin{align}\label{e.bxi^j=psi^j=}
    \bxi^j(x) =  2^{-j}\sum_{k=1}^{2^j}\xi\Ll(x_k\Rr),\quad \forall x\in \cL^j; \qquad \psi^j(x) = \psi\Ll(\sum_{k=1}^{2^j}x_k \mathds{1}_{\big[\frac{k-1}{2^j}, \frac{k}{2^j}\big)}\Rr),\qquad\forall x \in \mcl Q^j_2.
\end{align}
For every $j\in \N$, let 
\begin{align}\label{e.f_j=}
    \text{$f_j$ be the unique Lipschitz viscosity solution of $\HJ(\itr(\mcl Q^j_2), \bxi^j;\psi^j)$}
\end{align}
whose expression is
\begin{align*}
\begin{cases}
    \partial_t f_j - \bxi^j (\nabla f_j) = 0, \quad &\text{on }[0,\infty)\times \itr(\mcl Q^j_2),
    \\
    f_j(0,\cdot) = \psi^j, &\text{on }\itr(\mcl Q^j_2).
\end{cases}
\end{align*}
As before, since $f_j$ is finite-dimensional, we use $\nabla$ (instead of $\bnabla$) to denote the gradient operator.
The well-posedness of this equation is ensured by~\cite[Theorem~2.2]{chen2022hamilton} (based on~\cite{chen2022hamiltonCones}). Notice that there is no need to impose any boundary condition, thanks to the monotonicity of $\bxi^j$ and $\psi^j$ inherited from that of $\xi$ and $\psi$.
We can extend the domain of $f_j$ to $[0,\infty)\times \mcl Q^j_2$ via continuity.
We view $f_j$ as the approximation of the equation~\eqref{e.hj} restricted to step functions on the $2^j$-dyadic partition. The following is a restatement of \cite[Theorem~4.6~(1)]{chen2022hamilton}.

\begin{lemma}[\cite{chen2022hamilton} Finite-dimensional approximations]\label{l.f_j_approx}
We assume the setup in Theorem~\ref{t} and let $f$ be given as therein. For each $j\in\N$, let $f_j$ be given as in~\eqref{e.f_j=}. Then, $(f^\uparrow_j)_{j\in\N}$ converges to $f$ in the local uniform topology on $[0,\infty)\times \mcl Q_2$.
\end{lemma}

Here and henceforth, convergence in the local uniform topology means uniform convergence on every bounded metric ball. Here, $f^\uparrow_j$ is well-defined on $[0,\infty)\times \mcl Q_2$ due to $\pj_j(\mcl Q_2)=\mcl Q^j_2$ given in Lemma~\ref{l.proj_cones}.
We are ready to state the precise version of Result~\ref{r.2}.

\begin{theorem}[Envelope representation for limits of derivatives]\label{t2}
We assume the setup in Theorem~\ref{t}. For each $j\in\N$, let $f_j$ be given as in~\eqref{e.f_j=} and let $(t_j,x_j)\in (0,\infty)\times \itr(\mcl Q^j_2)$ be a differentiable point of $f_j$. Then, the sequence
\begin{align}\label{e.(Df_j...)}
    \Big(\,\partial_t f_j(t_j,x_j),\ \lf_j\big(\nabla f_j(t_j,x_j)\big)\,\Big)_{j\in\N}
\end{align}
is precompact in $\R\times L^2$.
Assume that there is $(t,q) \in [0,\infty)\times \mcl Q_2$ such that $(t_j,\lf_j x_j)_{j\in\N}$ converges in $\R\times L^2$ to $(t, q)$. Then, for any subsequential limit $(a,p)\in \R\times L^2$ of the sequence~\eqref{e.(Df_j...)}, there is a Borel probability measure $\gamma_{t,q}$ on $\mcl Q_2$ such that
\begin{gather}\label{e.t2}
\begin{split}
    f(t,{q}) - \la {q}, p\ra_\cH -ta = \int_{\C_2} \psi({q'}) -\la{q'}, \bnabla\psi({q'})\ra_{\cH}  \d \gamma_{t,{q}}({q'}),
    \\
    a = \int_{\C_2} \bxi\left(\bnabla\psi({q'})\right)\d \gamma_{t,{q}}({q'}),\qquad p = \int_{\C_2} \bnabla\psi({q'}) \d \gamma_{t,{q}}({q'}).
\end{split}
\end{gather}
In particular, since we always have $a=\bxi(p)$, the measure $\gamma_{t,q}$ also satisfies
\begin{align}\label{e.t2_particular}
    \int_{\C_2} \bxi\left(\bnabla\psi({q'})\right)\d \gamma_{t,{q}}({q'}) = \bxi\left(\int_{\C_2} \bnabla\psi({q'})\d \gamma_{t,{q}}({q'})\right).
\end{align}
\end{theorem}

\section{Analysis of approximations}\label{s.fin_d_PDE}

Throughout this section, we assume the setup in Theorem~\ref{t}. This section contains the main technical results.
In Section~\ref{s.basics}, we give more detailed results on projections and lifts. In Section~\ref{s.f-dPDEs}, we introduce more families of approximation equations and study their properties. In Section~\ref{s.other_prelim_results}, we prove other results needed in the next section.

\subsection{Basics}\label{s.basics}

Recall the linear space $\cL^j$ and its inner product in~\eqref{e.<x,y>_cL^j=}.

\subsubsection{Calculus on $\cL^j$}

For a function $g:\cL^j\to\R$, we write $\nabla g = (\partial_i g)_{i=1}^{2^j}$ where $\partial_i g$ is given by 
\begin{align*}
    g(x_1,\dots,x_{i-1},x_i+\eps a,x_{i+1},\dots,x_{2^j})-g(x) = \eps a\cdot \partial_i g(x) + o(\eps)
\end{align*}
for every $a\in\S^\D$ and $\eps>0$ small. Hence, we have $\partial_i g(x) \in \S^\D$ and $\nabla g(x)\in \cL^j$.

We want to express $\partial_i g$ in coordinates.
Fix an orthonormal basis $(e_k)_{k=1}^{D(D+1)/2}$ of $\S^\D$. Then, we set 
\begin{align*}
    \partial_{i,k}g(x) = \frac{\d}{\d \eps} g(x_1,\dots,x_{i-1},x_i+\eps e_k,x_{i+1},\dots,x_{2^j})\Big|_{\eps=0}
\end{align*}
and we can express
\begin{align*}
    \partial_i g(x) = (\partial_{i,k} g(x) e_k)_{k=1}^{D(D+1)/2}.
\end{align*}
Equivalently, $\partial_{i,k} = e_k\cdot\partial_i$.

We define the Laplacian and the divergence operator by
\begin{align}\label{e.Delta,div=}
    \Delta g= \frac{1}{2^j} \sum_{i=1}^{2^j}\sum_{k=1}^{D(D+1)/2}\partial_{i,k}^2g,\qquad \div v =\frac{1}{2^j} \sum_{i=1}^{2^j}\sum_{k=1}^{D(D+1)/2}\partial_{i,k}(v_i\cdot e_k)
\end{align}
where $v$ is a function from a subset of $\cL^j$ to $\cL^j$. Then, if $g,h, v$ are regular and decaying sufficiently fast, we have the usual integration by parts formulas
\begin{gather}\label{e.integration_by_parts}
    \int g\Delta h = -\int \la \nabla g,\nabla h\ra_{\cL^j},\qquad
    \int g \,\div v = -\int \la \nabla g, v\ra_{\cL^j}
\end{gather}
with respect to the $2^{j-1}D(D+1)$-dimensional Lebesgue measure on $\cL^j$. Notice the normalizing factor $\frac{1}{2^j}$ in $\Delta$ and $\div$, which is related to that in~\eqref{e.<x,y>_cL^j=}. 

\subsubsection{Projection and lift}

Recall projections and lifts defined in~\eqref{e.pj,lj=}.
For $\kappa\in\cH$, we define $\kappa^\j $ to be a locally-averaged approximation of $\kappa$ given by
\begin{gather}\label{e.kappa^(j)}
    \kappa^\j = \lf_j\pj_j\kappa = \sum_{k=1}^{2^j}\left(2^j\int_{\frac{k-1}{2^j}}^\frac{k}{2^j}\kappa(s)\d s\right) \mathds{1}_{\big[\frac{k-1}{2^j}, \frac{k}{2^j}\big)}.
\end{gather}
We recall some basic properties from \cite[Lemma~3.3]{chen2022hamilton}.

\begin{lemma}[\cite{chen2022hamilton} Basic properties of projects and lifts]\label{l.basics_(j)}
For every $j\in\N$, the following holds:
\begin{enumerate}

\item \label{i.isometric}  $\lf_j$ is isometric: $\la \lf_jx,\lf_jy\ra_{\cH}= \la x,y\ra_{\cL^j}$ for every $x,y\in\cL^j$;

\item \label{i.pj_lf=ID}
 $\pj_j\lf_j$ is the identity map on $\cL^j$:
$\pj_j\lf_j x=x$ for every $x\in\cL^j$;

\item \label{i.|iota^j|<|iota|}
$\pj_j$ is a contraction: $|\pj_j \kappa|_{\cL^j}\leq |\kappa|_\cH$, or equivalently $|\kappa^\j|_\cH\leq |\kappa|_\cH$, for every $\kappa\in\cH$;

\item \label{i.projective} 
if $j\leq j'$, then $\pj_j \lf_{j'}\pj_{j'}\kappa =\pj_j\kappa$ for every $\kappa \in \cH$.

\end{enumerate}
In addition, the following convergence result holds:
\begin{enumerate} \setcounter{enumi}{6}
    \item \label{i.cvg} for every $\kappa\in\cH$, $\lim_{j\to\infty}\kappa^\j =\kappa$ in $\cH$.

\end{enumerate}

\end{lemma}

Recall the definitions of projections and lifts of functions from~\eqref{e.j-projection} and~\eqref{e.lift=}.
Also recall $\mcl Q^j_2$ in~\eqref{e.Q^j_2=}.
We restate \cite[Lemma~3.6]{chen2022hamilton} below.

\begin{lemma}[\cite{chen2022hamilton} Projection and lifts of differentiable functions]\label{l.derivatives}
For every $j\in \N$, the following holds.
\begin{enumerate}
    \item \label{i.char_nabla_j} If $g:\C_2\to\R$ is differentiable at $\lf_j x$ for some $x\in \C_2^j$, then $g^j:\C_2^j\to\R$ is differentiable at $x$ and $\nabla g^j(x) = \pj_j(\bnabla g(\lf_j x))$.
    \item \label{i.lift_gradient}
If $g:\C_2^j\to\R$ is differentiable at $ x$ for some $x\in \C_2^j$, then $g^\uparrow:\C_2\to\R$ is differentiable at every ${q}\in\C_2$ satisfying $\pj_j{q} = x$ and $\bnabla g^\uparrow({q}) = \lf_j(\nabla g(x)) $.
\end{enumerate}
\end{lemma}

\subsubsection{Properties of cones}

Recall the definition of dual cones in Definition~\ref{d.dual_cone}. We often need the following classic result (for instance, see \cite[Corollary~6.33]{bauschke2011convex}).

\begin{lemma}[Duality of cones]\label{l.duality_cones}
Let $K$ be a nonempty closed convex cone in a Hilbert space $E$, then we have $(K^*)^*=K$.
\end{lemma}

The two lemmas below are restatements of \cite[Lemma~3.4 and Lemma~3.5]{chen2022hamilton}, respectively.

\begin{lemma}[\cite{chen2022hamilton} Characterizations of dual cones]\label{l.dual_cone}

\leavevmode
\begin{enumerate}
    \item For each $j\in\N$, the dual cone of $\C_2^j$ in $\cL^j$ is
\begin{align*}
    \Ll(\C_2^j\Rr)^* = \left\{x\in\cL^j: \sum_{i=k}^{2^j}x_i\in\S^\D_+,\quad\forall k\in\{1,2,\dots,2^j\}\right\}.
\end{align*}

    \item The dual cone of $\C_2$ in $\cH$ is
\begin{align*}
    \C_2^* = \left\{\kappa\in \cH: \int_t^1\kappa(s)\d s\in\S^\D_+,\quad \forall t\in [0,1) \right\}.
\end{align*}
\end{enumerate}
\end{lemma}

\begin{lemma}[\cite{chen2022hamilton} Projection and lifts of cones]\label{l.proj_cones}
For every $j\in \N$, we have
\begin{align*}
    \lf_j\Ll(\C_2^j\Rr)\subset \C_2,\quad \lf_j\Ll(\Ll(\C_2^j\Rr)^*\Rr)\subset \C_2^*,\quad \pj_j\Ll(\C_2\Rr)= \C_2^j,\quad \pj_j\Ll(\C_2^*\Rr)= \Ll(\C_2^j\Rr)^*.
\end{align*}
\end{lemma}

We close this subsection with results on the approximation and the precompactness of paths.

\begin{lemma}[Approximation of paths]\label{l.L^1cvg}
    For every $j\in\N$ and every ${q}\in \C_2$,
\begin{align*}
    \left|{q} - {q}^\j\right|_{L^1}\leq 2^{\frac{3-j}{2}}\D|{q}|_\cH.
\end{align*}
\end{lemma}

\begin{proof}
Setting $\delta = \frac{1}{2^j}$, we can compute
\begin{align*}
    \left|{q}-{q}^\j\right|_{L^1} = \sum_{k=1}^{\delta^{-1}}\int_{(k-1)\delta}^{k\delta}\left|{q}(s) - \delta^{-1}\int_{(k-1)\delta}^{k\delta}{q}(r) \d r\right|\d s \\
    \leq 2\delta^{-1} \sum_{k=1}^{\delta^{-1}}\int_{(k-1)\delta}^{k\delta}\int_{(k-1)\delta}^{s}\Ll|{q}(s)-{q}(r)\Rr|\d r\d s.
\end{align*}
Recall that here $|\cdot|$ is the norm associated with the entry-wise inner product on $\S^\D$. By diagonalization, it is straightforward to see $|a|\leq \sqrt{\D}\tr(a)\leq \D |a|$ for every $a\in\S^\D_+$. When $s\geq r$, by monotonicity (see~\eqref{e.mclQ=}), we have $q(s)-q(r)\in \S^\D_+$. Writing $\bar q =\tr(q)$, we thus have
\begin{align*}
    \left|{q}-{q}^\j\right|_{L^1} \leq 2\sqrt{\D}\delta^{-1} \sum_{k=1}^{\delta^{-1}}\int_{(k-1)\delta}^{k\delta}\int_{(k-1)\delta}^{s}\bar q(s)-\bar q(r)\d r\d s.
\end{align*}
For $B>0$ to be chosen, using the monotonicity of $\bar q$, we have
\begin{align*}
    \delta^{-1} \sum_{k=1}^{\delta^{-1}}\int_{(k-1)\delta}^{k\delta}\int_{(k-1)\delta}^{s}(\bar q(s)-\bar q(r))\mathds{1}_{\bar q(s)>B}\d r\d s\leq \sum_{k=1}^{\delta^{-1}}\int_{(k-1)\delta}^{k\delta}\bar q(s)\mathds{1}_{\bar q(s)>B} \d s\\
    = \int_0^1\bar q(s)\mathds{1}_{\bar q(s)>B} \d s\leq \frac{1}{B}\int_0^1 \bar q(s)^2\d s \leq \frac{\D}{B}| q|^2_\cH.
\end{align*}
For the other half, changing variables and interchanging summations, we have
\begin{align*}
    &\delta^{-1} \sum_{k=1}^{\delta^{-1}}\int_{(k-1)\delta}^{k\delta}\int_{(k-1)\delta}^{s}(\bar q(s)-\bar q(r))\mathds{1}_{\bar q(s)\leq B}\d r\d s
    \\
    &= \delta^{-1} \int_{0}^{\delta}\int_{0}^{s}\d r\d s\sum_{k=1}^{\delta^{-1}}(\bar q((k-1)\delta+s)-\bar q((k-1)\delta+r))\mathds{1}_{\bar q((k-1)\delta+s)\leq B}.
\end{align*}
Setting $\kappa(s) = \max\{k:\bar q((k-1)\delta+s)\leq B\}$. Then, the summation above is equal to
\begin{align*}
    \sum_{k=1}^{\kappa(s)}(\bar q((k-1)\delta+s)-\bar q((k-1)\delta+r))\leq \bar q((\kappa(s)-1)\delta+s) \leq B
\end{align*}
where in the first inequality, we used the fact that $-\bar q((k-1)\delta+r) + \bar q((k-2)\delta + s)\leq 0$ due to $s-r\leq \delta$, and that $\bar q\geq 0$.
Hence, we obtain
\begin{align*}
    \delta^{-1} \sum_{k=1}^{\delta^{-1}}\int_{(k-1)\delta}^{k\delta}\int_{(k-1)\delta}^{s}(\bar q(s)-\bar q(r))\mathds{1}_{\bar q(s)\leq B}\d r\d s \leq \delta^{-1} \int_{0}^{\delta}\int_{0}^{s}B\d r\d s = \frac{\delta B}{2}.
\end{align*}
In conclusion, we have $| q- q^\j|_{L^1} \leq \frac{2D^\frac{3}{2}}{B}|q|^2_\cH + \delta D^\frac{1}{2} B$. Optimizing over $B$, we get the desired result.
\end{proof}

\begin{lemma}[Upgrade of convergence]\label{l.cvg_L^r}
Let $({p}_n)_{n\in\N}$ be a sequence in $\C_\infty$ that satisfies $\sup_{n\in\N}|p_n|_{L^\infty}<\infty$ and converges weakly in $\cH$ to some ${p}_\infty$. Then, ${p}_n$ converges to ${p}_\infty$ in $L^r$ for every $r\in[1,\infty)$.
\end{lemma}
\begin{proof}
Due to~$\rv{\sup_{n\in\N}|p_n|_{L^\infty}<\infty}$, by re-scaling, we can assume $|p_n|_{L^\infty}\leq 1$ for every $n$.
The weak convergence implies $|{p}_\infty|_\cH\leq 1$. Lemma~\ref{l.L^1cvg} yields
\begin{align}\label{e.cvg.zeta_n-zeta^j_n}
    \sup_{n\in\N\cup\{\infty\}}\Ll|{p}_n - {p}^\j_n\Rr|_{L^1} \leq 2^{\frac{3-j}{2}}\D\Ll|{p}_n\Rr|_\cH\leq 2^{\frac{3-j}{2}}\D,\quad\forall j\in\N.
\end{align}
Since $\Ll({p}^\j_n\Rr)_{n\in\N\cup\{\infty\}}$ are step functions on a fixed partition of $[0,1)$, we can map them isometrically and linearly into a finite-dimensional Euclidean space. Hence, for each $j\in\N$, we obtain from the weak convergence that
\begin{align}\label{e.cvg.zeta^j_n-zeta^j_infty}
    \lim_{n\to\infty} \Ll|{p}^\j_n-{p}^\j_\infty\Rr|_\cH =0,\quad\forall j\in\N.
\end{align}
For every $\eps>0$, we first choose $j$ sufficiently large so that $2^{\frac{3-j}{2}}\D\leq \frac{\eps}{4}$ and then use~\eqref{e.cvg.zeta^j_n-zeta^j_infty} to choose $n_0$ sufficiently large such that $\Ll|{p}^\j_n-{p}^\j_\infty\Rr|_\cH\leq \frac{\eps}{2}$ for all $n\geq n_0$. Then, by~\eqref{e.cvg.zeta_n-zeta^j_n} and the triangle inequality, we have
\begin{align*}
    \Ll|{p}_n-{p}_\infty\Rr|_{L^1} \leq \Ll|{p}_n-{p}^\j_n\Rr|_{L^1} + \Ll|{p}^\j_n-{p}^\j_\infty\Rr|_{L^1}+ \Ll|{p}^\j_\infty-{p}_\infty\Rr|_{L^1}\leq \eps
\end{align*}
for all $n\geq n_0$. Therefore, ${p}_n$ converges to ${p}_\infty$ in $L^1$.
We can choose a subsequence along which ${p}_n$ converges to ${p}_\infty$ pointwise a.e. This implies $|{p}_\infty|_{L^\infty}\leq 1$, which together with the uniform $L^\infty$-bound on ${p}_n$ upgrades the $L^1$ convergence to $L^r$ for any $r\in[1,\infty)$.
\end{proof}

\subsection{Finite-dimensional PDEs}\label{s.f-dPDEs}

We need families of finite-dimensional PDEs to approximate the solution $f$ of $\HJ(\mcl Q_2,\H;\psi)$. We start with the one already introduced in~\eqref{e.f_j=}.

\subsubsection{Bounds on derivatives of $f_j$.}

To show the precompactness of~\eqref{e.(Df_j...)} in Theorem~\ref{t2}, we need the following result. Recall $\mcl Q$ from~\eqref{e.mclQ=}, $L^\infty$ from~\eqref{e.L^p}, and the lift of functions in~\eqref{e.lift=}. 

\begin{lemma}\label{l.bound_der_f_j}
Let $j\in\N$, let $f_j$ be given as in~\eqref{e.f_j=}, and let $(t_j,x_j)\in(0,\infty)\times \itr(\mcl Q^j_2)$ be any differentiable point of $f_j$. Then, we have
\begin{align}\label{e.l.bound_der_f_j}
    \Ll|\partial_t f_j(t_j,x_j)\Rr|\leq \sup_{a\in \S^\D_+:\: |a|\leq1}\Ll|\xi(a)\Rr|,\qquad \lf_j\Ll(\nabla f_j(t_j,x_j)\Rr) \in \mcl Q,\qquad \Ll|\lf_j\Ll(\nabla f_j(t_j,x_j)\Rr)\Rr|_{L^\infty}\leq 1.
\end{align}
\end{lemma}

\begin{proof}
\rv{Recall that the construction of $f_j$ was carried out in~\cite[Theorem~2.2]{chen2022hamilton}. Consequently, $f_j$ belongs to the class described in~\cite[(2.1)]{chen2022hamilton}, which implies that $f_j(t,\cdot)$ is $(\mcl Q^j_2)^*$-nondecreasing (see Definition~\ref{d.dual_cone}) for each $t \in [0,\infty)$.}
Using this and Lemma~\ref{l.duality_cones}, we can deduce $\nabla f_j(t_j,x_j) \in \mcl Q^j_2$. Applying Lemma~\ref{l.proj_cones}, we get $\lf_j(\nabla f_j(t_j,x_j) ) \in \mcl Q_2$. Now, we turn to the bounds on derivatives.
On $\cL^j$, we consider the $\ell^1$- and the $\ell^\infty$-norms defined by
\begin{align*}
    |x|_{\ell^1} = 2^{-j}\sum_{k=1}^{2^j}|x_k|,\qquad |x|_{\ell^\infty} = \sup_{1\leq k\leq2^j}|x_k|
\end{align*}
where $x=(x_k)_{1\leq k\leq 2^j}\in\cL^j$ with each $x_k \in \S^\D$.
Applying~\cite[Proposition~2.4]{chen2022hamilton} with $\|\cdot\|_p$ therein replaced by $|\cdot|_{\ell^1}$, we get
\begin{align*}
    \Ll|\nabla f_j(t_j,x_j)\Rr|_{\ell^\infty} \leq C,\qquad \Ll|\partial_t f_j(t_j,x_j)\Rr|\leq\sup_{x\in \mcl Q^j_2,\, |x|_{\ell^\infty}\leq  C}\Ll|\bxi^j(x)\Rr|
\end{align*}
where $C= \sup_{x\in\mcl Q^j_2}\Ll|\nabla \psi^j(x)\Rr|_{\ell^\infty}$.
Using~\ref{i.A} and Lemma~\ref{l.derivatives}~\eqref{i.char_nabla_j}, we can see that $C\leq 1$. Using this and the expression of $\bxi^j$ in~\eqref{e.bxi^j=psi^j=}, we can get the first bound in~\eqref{e.l.bound_der_f_j}. By the definition of $\lf_j$ in~\eqref{e.pj,lj=}, we get the last bound in~\eqref{e.l.bound_der_f_j}. 
\end{proof}

\subsubsection{Extension lemma}

For each $j$, we need to consider PDEs posed on the whole space $[0,\infty)\times \cL^j$.
Hence, we first need to extend the initial condition $\psi^j$ to $\cL^j$.

\begin{lemma}[$C^1$-extension on cones]\label{l.ext}
    Suppose that $u:\C_2^j\to\R$ is continuously differentiable and that $\nabla u:\C_2^j\to \cL^j$ is bounded and Lipschitz. Then, there is a continuously differentiable function $\tilde u:\cL^j\to\R$ such that $\tilde u |_{\C_2^j} = u$ and that $\nabla\tilde u$ is bounded and Lipschitz.
\end{lemma}

This lemma is a version of the Whitney extension theorem. To our best knowledge, there are no existing results on extending a possibly unbounded function on an unbounded domain but with bounded derivatives. For instance, the original version~\cite{whitney} gives an extension on unbounded domains but without bounds on derivatives; \cite[Theorem~2.3.6]{hormander2015analysis} gives an extension on a compact domain; and \cite[Theorem~3 in Section~2.2.1 of Chapter~VI]{stein1970singular} gives an extension of a bounded function on an unbounded domain.
We prove this lemma by modifying the arguments in the proof of \cite[Theorem~2.3.6]{hormander2015analysis} and exploiting the linear structure of a cone.

We mention that Lemma~\ref{l.ext} does not give a good bound on $\nabla \tilde u$ in terms of $j$. This causes some difficulty to us because we want to apply this lemma to extend $\psi^j$ (see~\eqref{e.bxi^j=psi^j=}) whose Lipschitz coefficient is bounded uniformly in $j$ but the lemma does not give us a control on that of $\tilde{\psi^j}$ in $j$. To overcome this, we will need Lemma~\ref{l.supp_cone} which states that everything interesting to us eventually happens on the cone.

\begin{proof}
\rv{Let $C > 0$ denote a constant depending only on the Lipschitz constants of $u$ and $\nabla u$, whose value may change from line to line.}
We denote the Euclidean distance between sets $A,B\subset\cL^j$ by $d(A,B)$. We write $d(x,y) = d(\{x\},\{y\})$, $d(x,A)=d(\{x\},A)$, and $d(x) = d(\{x\},\C_2^j)$. 

The support of a function $\phi$ on $\cL^j$ is denoted by $\mathsf{supp}\,\phi$ whose diameter is denoted by $\mathsf{diam}\,\phi$. Also, we write $d(\phi, *) =  d(\mathsf{supp}\,\phi,*)$ for $*$ being either a point or a set.

We need the existence of a partition of unity: there is a countable collection $(\phi_k)_{k\in \N}$ of continuously differentiable functions from $\cL^j$ to $[0,\infty)$ such that
\begin{enumerate}[start=1,label={\rm (p\arabic*)}]
    \item \label{i.phi_1} $\sum_{k\in\N}\phi_k=1$ on $\cL^j\setminus\C_2^j$;
    \item \label{i.phi_2} there  is $n\in\N$ such that no point in $\cL^j$ is the support of more than $n$ functions;
    \item \label{i.phi_3} $\mathsf{diam}\,\phi_k \leq 2 d(\phi_k,\C_2^j)$ for every $k$;
    \item \label{i.phi_4} there is $C_0>0$ such that $|\nabla\phi_k(x)|\leq C_0d(x)^{-1}$ for every $k$ and every $x\in\cL^j\setminus\C_2^j$.
\end{enumerate}
This is a restatement of \cite[Lemma~2.3.7]{hormander2015analysis}, which is a direct consequence of \cite[Theorem~1.4.10 and Example~1.4.8]{hormander2015analysis}. As an intermediate step in the proof of \cite[Theorem~2.3.6]{hormander2015analysis}, \cite[Lemma~2.3.7]{hormander2015analysis} is stated for a compact set in place of $\C_2^j$. However, examining the conditions in \cite[Theorem~1.4.10 and Example~1.4.8]{hormander2015analysis}, one can see that compactness is superfluous.

Let us denote the inner product (see~\eqref{e.<x,y>_cL^j=}) in $\cL^j$ simply by a dot (this is a slight abuse of notation, since elsewhere the dot product is the entry-wise inner product without the normalizing factor $2^{-j}$ in~\eqref{e.<x,y>_cL^j=}).
We set
\begin{align}\label{e.u(x,y)=u(y)+(x-y)nablau(y)}
    u(x,y) = u(y)+ (x-y)\cdot \nabla u(y),\quad\forall x \in \cL^j,\ y\in \C_2^j.
\end{align}
Hence, $u(\cdot,y)$ is the first-order Taylor polynomial of $u$ expanded at $y$. Let us record some basic estimates of $u(\cdot,\cdot)$. Since the \rv{boundedness of $\nabla u$} implies $|u(y)|\leq C(1+|y|)$, together with the triangle inequality, we \rv{can obtain from~\eqref{e.u(x,y)=u(y)+(x-y)nablau(y)} that}
\begin{align}\label{e.|u|}
    |u(x,y)|\leq C(1+|x|+|x-y|),\quad\forall x\in\cL^j,\  y\in\C_2^j.
\end{align} Boundedness of $\nabla u$ immediate gives
\begin{align}\label{e.|u-u|<1}
    |u(x,y)-u(x,y')| \leq C(|x-y|+|x-y'|),\quad \forall x\in \cL^j,\ y,y'\in\C_2^j.
\end{align}
Using
\begin{align*}
    u(x,y)-u(x,y') = u(y)-u(y')-(y-y')\cdot\nabla u(y')+(x-y)\cdot(\nabla u(y)-\nabla u(y')),
\end{align*}
the Lipschitzness of $\nabla u$, and the triangle inequality, we can obtain
\begin{align}\label{e.|u-u|<}
    |u(x,y)-u(x,y')|\leq C(|x-y|^2+|x-y'|^2),\quad \forall x\in \cL^j,\ y,y'\in\C_2^j.
\end{align}

For each $k$, we fix $y_k\in\C_2^j$ to satisfy $d(\phi_k,y_k)=d(\phi_k,\C_2^j)$. For $x \in \mathsf{supp}\,\phi_k$, using $|x-y_k|\leq d(\phi_k,\C_2^j)+\mathsf{diam}\,\phi_k$ and~\ref{i.phi_3}, we have $|x-y_k|\leq 3 d(\phi_k,\C_2^j)$ and thus
\begin{align}\label{e.|x-y|<3}
    |x-y_k|\leq 3|x-y|,\quad\forall x\in \mathsf{supp}\,\phi_k,\ y\in \C_2^j.
\end{align}

Define the following sets:
\begin{gather*}
    X_0=\{x\in\cL^j:|x|\leq 1\},\qquad X_1 = \{x\in\cL^j:d(x)\leq 4^{-1}|x|\},\qquad X= X_0\cup X_1;
    \\
    K_0 = \{k:d(\phi_k,0)\leq 1\},\qquad K_1 = \{k:\mathsf{diam}\,\phi_k\leq d(\phi_k,0)\},\qquad K=K_0\cup K_1.
\end{gather*}
A simple observation from the definitions of $X_1$ and $X$ is 
\begin{align}\label{e.1<|x|<4d(x)}
    x\not\in X\quad \Longrightarrow \quad 1\leq |x|\leq 4d(x).
\end{align}
For $i\in\{0,1\}$, we want to show 
\begin{align}\label{e.if_then}
    x\in X_i \cap  \mathsf{supp}\,\phi_k\quad\implies \quad k\in K_i.
\end{align}
For $i=0$, this follows from $d(\phi_k,0)\leq |x|$ if $x\in \mathsf{supp}\, \phi_k$. For $i=1$, we have
\begin{align*}
    \mathsf{diam}\,\phi_k \stackrel{\ref{i.phi_3}}{\leq} 2d(\phi_k,\C_2^j)\stackrel{x\in \supp\phi_k}{\leq} 2d(x)\stackrel{x\in X_1}{\leq} \frac{1}{2}|x|\stackrel{x\in \supp\phi_k}{\leq} \frac{1}{2}(\mathsf{diam}\,\phi_k+d(\phi_k,0))
\end{align*}
which implies $k\in K_1$. Hence, we have verified~\eqref{e.if_then} for $i\in\{0,1\}$ and consequently
\begin{gather}\label{e.sumphi=1}
    \sum_{k\in K}\phi_k(x) \stackrel{\ref{i.phi_1}\eqref{e.if_then}}{=}1,\quad\forall x \in X;
    \\
    \label{e.sumgradphi=0}
    \sum_{k\in K}\nabla\phi_k(x) \stackrel{\eqref{e.sumphi=1}}{=}0,\quad\forall x \in X.
\end{gather}
Also, note that the interior of $X$ contains $\C_2^j$.

Then, we set
\begin{align}\label{e.v(x)=}
    v(x) = 
    \begin{cases}
    \sum_{k\in K}\phi_k(x)u(x,y_k),\quad & x\not\in\C_2^j,
    \\
    u(x),\quad & x\in \C_2^j.
    \end{cases}
\end{align}
We want to show that $v$ is the desired extension. Clearly, $v$ is continuously differentiable in $\cL^j\setminus\C_2^j$ and the interior of $\C_2^j$. 

First, we show that $v$ is differentiable on $\partial\C_2^j$ and $\nabla v = \nabla u$ on $\partial\C_2^j$. Let $y\in \partial \C_2^j$. Since $v(x)-v(y)-(x-y)\cdot \nabla u(y)=v(x)-u(x,y)$ and the interior of $X$ contains $\C_2^j$, it suffices to show that
\begin{align}\label{e.|v-u|<}
    |v(x) - u(x,y)|\leq C|x-y|^2,\quad\forall x\in  X.
\end{align}
If $x\in \C_2^j$, then this holds due to the regularity of $u$. Now, suppose $x \in \cL^j\setminus\C_2^j$. Then, we have
\begin{align*}
    |v(x) - u(x,y)|\stackrel{\eqref{e.sumphi=1}}{\leq}\sum_k\phi_k(x)|u(x,y_k)-u(x,y)|\stackrel{\eqref{e.|u-u|<}\eqref{e.|x-y|<3}}{\leq} \sum_k\phi_k(x)C|x-y|^2
    \\
    \stackrel{\eqref{e.sumphi=1}}{=} C|x-y|^2.
\end{align*}
Hence, we have verified~\eqref{e.|v-u|<} and thus $v$ is differentiable  on $\partial\C_2^j$ with $\nabla v=\nabla u$ on $\partial \C_2^j$. In particular, $v$ is differentiable everywhere.

Next, we show that $\nabla v$ is Lipschitz.
In view of the definition of $v$ in~\eqref{e.v(x)=} and~\ref{i.phi_2}, $\nabla v$ is Lipschitz on $\C_2^j$ and outside of $\C_2^j$. To have the full Lipschitzness on $\cL^j$, it suffices to show $|\nabla v(x) -\nabla v(y)|\leq C|x-y|$ for $x\in X\setminus \C_2^j$ and $y\in \partial \C_2^j$ (recall that the interior of $X$ contains $\C_2^j$). 
Firstly, notice
\begin{align}\label{e.nabla_v(x)=}
    \nabla v(x) \stackrel{\eqref{e.u(x,y)=u(y)+(x-y)nablau(y)}\eqref{e.v(x)=}}{=} \sum_{k\in K}\phi_k(x)\nabla u(y_k) + \sum_{k\in K}\nabla\phi_k(x) u(x,y_k),\quad\forall x \in \cL^j\setminus\C_2^j.
\end{align}
Then, due to~\eqref{e.nabla_v(x)=}, $\nabla v(y) =\nabla u(y)$, and~\eqref{e.sumphi=1}, we have
\begin{align*}
    |\nabla v(x)-\nabla v(y)|\leq \sum_{k\in K}\phi_k(x)|\nabla u(y_k)-\nabla u(y)|+  \underbrace{\Ll|\sum_{k\in K}\nabla\phi_k(x) u(x,y_k)\Rr|}_\circledast,\qquad \forall x\in X\setminus\C_2^j.
\end{align*}
Using the Lipschitzness of $\nabla u$, we have
\begin{align*}
    \sum_{k\in K}\phi_k(x)|\nabla u(y_k)-\nabla u(y)|\leq \sum_{k\in K}\phi_k(x)C(|x-y_k|+|x-y|)\stackrel{\eqref{e.|x-y|<3}\eqref{e.sumphi=1}}{\leq} C|x-y|.
\end{align*}
Since $\C_2^j$ is closed and convex, there is $y^x\in\C_2^j$ such that $d(x)=|x-y^x|$. 
Then, we get
\begin{align*}
    \circledast \stackrel{\eqref{e.sumgradphi=0}}{=}\Ll|\sum_{k\in K}\nabla\phi_k(x) (u(x,y_k)-u(x,y^x))\Rr|  \leq \sum_{k\in K}|\nabla\phi_k(x)||u(x,y_k)-u(x,y^x)|
    \\
    \stackrel{\eqref{e.|u-u|<}\eqref{e.|x-y|<3}\ref{i.phi_2}\ref{i.phi_4}}{\leq} Cn d(x)^{-1}|x-y^x|^2 \stackrel{d(x)=|x-y^x|}{=} Cn d(x) \stackrel{y\in \partial \C_2^j}{\leq} Cn|x-y|.
\end{align*}
Therefore, we obtain $|\nabla v(x)-\nabla v(y)|\leq C|x-y|$, which as we explained previously implies the full Lipschitzness of $v$ on $\cL^j$.

Lastly, we show that $\nabla v$ is bounded. We only need to verify this on $\cL^j\setminus\C_2^j$. In view of~\eqref{e.nabla_v(x)=}, since $\nabla u$ is bounded,
we only need to bound $\circledast$ for all $x\in \cL^j\setminus\C_2^j$.
On $X\setminus \C_2^j$, we have
\begin{align*}
    \circledast &\stackrel{\eqref{e.sumgradphi=0}}{=}\Ll|\sum_{k\in K}\nabla\phi_k(x) (u(x,y_k)-u(x,y^x))\Rr| 
    \leq \sum_{k\in K}|\nabla\phi_k(x)||u(x,y_k)-u(x,y^x)|
    \\
    & \stackrel{\eqref{e.|u-u|<1}\eqref{e.|x-y|<3}}{\leq} \sum_{k\in K}|\nabla\phi_k(x)|C|x-y^x|
    \stackrel{\ref{i.phi_2}\ref{i.phi_4}}{\leq} Cn d(x)^{-1}|x-y^x| = Cn .
\end{align*}
On $\cL^j\setminus X$, we have
\begin{align*}
    \circledast &\leq \sum_{k\in K}|\nabla \phi_k(x)||u(x,y_k)| 
    \stackrel{\eqref{e.|u|}}{\leq} \sum_{k\in K}|\nabla \phi_k(x)|C\Ll(1+|x|+|x-y_k|\Rr)
    \\
    &\stackrel{\eqref{e.|x-y|<3}}{\leq} \sum_{k\in K}|\nabla \phi_k(x)|C\Ll(1+|x|+|x-y^x|\Rr)
    \stackrel{\eqref{e.1<|x|<4d(x)}}{\leq} \sum_{k\in K}|\nabla \phi_k(x)|C d(x) \stackrel{\ref{i.phi_2}\ref{i.phi_4}}{\leq} Cn.
\end{align*}
Therefore, $\nabla v$ is bounded. Taking $\tilde u=v$ completes the proof.
\end{proof}

\subsubsection{First-order approximation PDEs}

Recall $\H$ defined in~\eqref{e.H=}. For each $j\in\N$, let $\H^j:\cL^j\to\R$ be the $j$-projection of $\H$ (see~\eqref{e.j-projection}). Recall from~\ref{i.d2} that $\H$ is Lipschitz. So, due to Lemma~\ref{l.basics_(j)}~\eqref{i.isometric}, $\H^j$ is Lipschitz with the Lipschitz coefficient bounded uniformly in $j$. 
For technical reasons, we want to mollify $\H^j$ to make it smooth. For this purpose, we let $\phi_{j,1}:\cL^j\to [0,\infty)$ be a smooth bump function with compact support and $\int \phi_{j,1}=1$. For each $\eta>0$, we define $\phi_{j,\eta} = \eta^{d_j}\phi_{j,1}(\eta^{-1}\,\cdot\,)$ where $d_j$ is the dimension of $\cL^j$. We define
\begin{align}\label{e.H^j_eta=}
    \H^j_\eta = \H^j * \phi_{j,\eta}
\end{align}
where $*$ is the convolution operator. 
We set $\H^j_0= \H^j$.
It is clear that
\begin{align}\label{e.bound_on_nabla_F_j}
    \sup_{j\in\N,\, \eta >0} \Ll|\nabla\H^j_\eta \Rr|_{\cL^j} \leq \|\H\|_\mathrm{Lip}.
\end{align}
This bound also allows us to see that, for each $j\in\N$,
\begin{align}\label{e.lim|F_j-H^j|=0}
    \lim_{\eta\to 0} \Ll\|\H^j_\eta - \H^j\Rr\|_\infty =0
\end{align}
where $\|\cdot\|_\infty$ is the standard sup-norm.

Let $\psi^j:\C_2^j \to \R$ be the $j$-projection of $\psi$ (introduced in~\eqref{e.bxi^j=psi^j=}), let $\tilde{\psi^j}:\cL^j\to\R$ be the extension of $\psi^j$ given by Lemma~\ref{l.ext} (which is applicable due to~\ref{i.A} and Lemma~\ref{l.derivatives}~\eqref{i.char_nabla_j}), and for each $\eta\geq0$, let
\begin{align}\label{e.f_j,eta=}
    \text{$f_{j,\eta}$ be the unique viscosity solution of $\HJ\Ll(\cL^j,\H^j_\eta; \tilde{\psi^j}\Rr)$}.
\end{align}
Since $\cL^j$ is isometric to a Euclidean space, the well-posedness of $\HJ\Ll(\cL^j,\H^j_\eta; \tilde{\psi^j}\Rr)$ is standard and can be seen, for instance, in \cite{crandall1992user}.
Recall $f_j$ introduced in~\eqref{e.f_j=}.

\begin{lemma}\label{l.f_j,0_cvg_f_j}
For each $j \in \N$, $f_{j,0}$ as in~\eqref{e.f_j,eta=} is equal to $f_j$ as in~\eqref{e.f_j=} everywhere on $[0,\infty)\times \mcl Q^j_2$.
\end{lemma}

\begin{proof}
In the statement of \cite[Theorem~4.8~(1)]{chen2022hamilton}, $f_j$ is taken as the unique solution of $\HJ\Ll(\itr(\C^j_2),\bxi^j;\psi^j\Rr)$. But examining the proof right after the statement, we can see that $f_j$ is initially taken to solve $\HJ\Ll(\C^j_2,\H^j;\psi^j\Rr)$ (then verified to solve the aforementioned equation) and thus also solves $\HJ\Ll(\itr(\C^j_2),\H^j;\psi^j\Rr)$, which by the comparison principle~\cite[Proposition~3.1]{chen2022hamiltonCones} coincides with $f_{j,0}|_{[0,\infty)\times \itr(\C^j_2)}$ here since $f_{j,0}|_{[0,\infty)\times \itr(\C^j_2)}$ and $f_j$ share the same initial condition. The conclusion follows by continuity. 
\end{proof}

\begin{lemma}\label{l.f_j,eta_cvg_f_j,0}
For each $j\in\N$ and $\eta\geq 0$, let $f_{j,\eta}$ be given as in~\eqref{e.f_j,eta=}. Then, $(f_{j,\eta})_{\eta>0}$ converges to $f_{j,0}$ in the local uniform topology on $[0,\infty)\times \cL^j$ as $\eta\to0$.
\end{lemma}

\begin{proof}
Since $\tilde {\psi^j}$ is Lipschitz and each of $(\H^j_\eta)_{\eta\in(0,1)}$ is locally bounded uniformly in $\eta$, the functions in $(f_{j,\eta})_{\eta\in(0,1)}$ are Lipschitz uniformly in $\eta$. Then, we can use \cite[Theorem~5.2.5]{cannarsa2004semiconcave} together with~\eqref{e.lim|F_j-H^j|=0} and the Arzelà--Ascoli theorem to deduce the convergence.
\end{proof}

\subsubsection{Viscous approximations and the adjoint}

This part follows Evans' work~\cite{evans2014envelopes} and uses techniques that have already appeared in~\cite{evans2010adjoint,tran2011adjoint,cagnetti2011aubry}. The main difference is that we do not send $\eps\to 0$ at this stage.
We need approximations of $f_{j,\eta}$ as in~\eqref{e.f_j,eta=} by solutions of a viscous equation.
For $j\in\N$, $\eta>0$, and $\eps>0$, let $f_{j,\eta,\eps}:[0,\infty)\times \cL^j\to\R$ be the unique viscosity solution of 
\begin{align}\label{e.f_j,eta,eps=}
    \begin{cases}
    \partial_t f_{j,\eta,\eps} -\H^j_\eta\Ll(\nabla f_{j,\eta,\eps}\Rr) =\eps \Delta f_{j,\eta,\eps} \qquad &\text{on $(0,\infty)\times \cL^j$},
    \\
    f_{j,\eta,\eps}(0,\cdot) =\tilde{\psi^j} \qquad &\text{on $\cL^j$}.
    \end{cases}
\end{align}

\begin{lemma}\label{l.f_j,eta,eps_to_f_j,eta}
For each $j\in\N$ and $\eta>0$, $(f_{j,\eta,\eps})_{\eps >0}$ as in~\eqref{e.f_j,eta,eps=} converges to $f_{j,\eta}$ as in~\eqref{e.f_j,eta=} in the local uniform topology on $[0,\infty)\times \cL^j$ as $\eps\to0$.
\end{lemma}

\begin{proof}
This is a result of \cite[Theorem~2.2]{evans2010adjoint}. There, the initial condition was assumed to be compactly supported. But examining the proof, we can see that this is not needed here. Also see \cite[Theorem~2.1]{tran2011adjoint}.
\end{proof}

Then, we define the adjoint of the linearization of $f_{j,\eta,\eps}$. 
For $j\in\N$, $\eta>0$, $\eps>0$, $(\underline t, \underline x)\in(0,\infty)\times \cL^j$, and $r>0$, let $\sigma_{j,\eta,\eps,(\underline t,\underline x),r}:[0,\underline t]\times\cL^j\to\R$ be the unique viscosity solution of
\begin{align}\label{e.sigma_PDE}
    \begin{cases}
    -\partial_t \sigma_{j,\eta,\eps,(\underline t,\underline x),r} + \div\Ll( \sigma_{j,\eta,\eps,(\underline t,\underline x),r}\nabla \H^j_\eta(\nabla f_{j,\eta,\eps})\Rr) =\eps \Delta \sigma_{j,\eta,\eps,(\underline t,\underline x),r} \qquad &\text{on $[0,\underline t)\times \cL^j$},
    \\
    \sigma_{j,\eta,\eps,(\underline t,\underline x),r} (\underline t,\cdot)=\frac{1}{|B_j(\underline x,r)|}\mathds{1}_{B_j(\underline x,r)} \qquad &\text{on $ \cL^j$}.
    \end{cases}
\end{align}
Here, $\Delta$ and $\mathsf{div}$ are the normalized ones in~\eqref{e.Delta,div=} and 
\begin{align}\label{e.B_j(x,r)=}
    B_j(\underline x, r)=\Ll\{x' \in \cL^j:\: \Ll|x'-\underline{x}\Rr|_{\cL^j}\leq r\Rr\}.
\end{align}
We set
\begin{align}\label{e.C()}
    C_j(\underline t,\underline x,r) = [\underline t, \underline t+r]\times B_j(\underline x, r)
\end{align}
and define the measure $\bsigma^{\eta,\eps}_{j,(\underline t,\underline x),r} $ on $\cL^j$ equipped with the Borel sigma-algebra by
\begin{align}\label{e.sigma^eps_j,(t,x),r}
    \d\bsigma^{\eta,\eps}_{j,(\underline t,\underline x),r} (x) = \fint_{\underline t}^{\underline t+r}\sigma_{j,\eta,\eps,(s,\underline x),r}(0,x)\d s\, \d x
\end{align}
where $\fint_{\underline t}^{\underline t+r} = \frac{1}{r}\int_{\underline t}^{\underline t+r}$.

\begin{lemma}\label{l.f,sigma_estimates}
For each $j\in\N$ and each $\underline{t}>0$, there is a constant $C_{j,\underline{t}}$ locally bounded in $\underline{t}$ such that the following holds for sufficiently small $\eta>0$ and $\eps>0$ and uniformly in $\underline{x} \in \cL^j$ and $r>0$:
\begin{gather}\
\eps \int_0^{\underline{t}}\int_{\cL^j}\Ll(\Ll|\nabla^2 f_{j,\eta,\eps}\Rr|^2+ |\nabla \partial_t f_{j,\eta,\eps}|^2\Rr)\sigma_{j,\eta,\eps,(\underline{t},\underline{x}),r} \d x \d t \leq C_{j,\underline{t}}; \label{e.nabla^2f_bound}
    \\
    \sigma_{j,\eta,\eps,(\underline{t},\underline{x}),r}\geq 0;\qquad \int_{\cL^j} \sigma_{j,\eta,\eps,(\underline{t},\underline{x}),r}(t,x) \d x =1,\quad\forall t\in [0,a). \label{e.sigma^a>0}
\end{gather}
\end{lemma}

\begin{proof}
These are from \cite[(2.6) and (2.7)]{evans2014envelopes} (see also \cite[(1.3) and Theorem~2.1]{evans2010adjoint} and \cite[Lemma~2.3 and Lemma~2.5]{tran2011adjoint}). They are stated without the additional approximation in $\eta$ here. But they remain valid for small $\eta>0$ due to~\eqref{e.bound_on_nabla_F_j}, \eqref{e.lim|F_j-H^j|=0}, and the fact that the initial condition for $f_{j,\eta,\eps}$ as in~\eqref{e.f_j,eta,eps=} remains fixed.
\end{proof}

The second relation in~\eqref{e.sigma^a>0} implies that $\bsigma^{\eta,\eps}_{j,(\underline{t},\underline{x}),r}$ is a probability measure. The next result is the main ingredient to prove Theorems~\ref{t} and~\ref{t2}.

\begin{lemma}[Envelope representation for approximations]\label{l.visc_app}
For every $R>0$, the following holds for all $j\in\N$, $\eta>0$, $\eps>0$, $(\underline t,\underline x)\in (0,R)\times\C_2^j$, and $r\in(0,R)$:
\begin{align}\label{e.l.visc_app}
\begin{cases}
\begin{aligned}
    \fint_{C_j(\underline t ,\underline x ,r)} f_{j,\eta,\eps} -  \la x,\nabla f_{j,\eta,\eps}\ra_{\cL^j} & - t \partial_tf_{j,\eta,\eps}  
    \\
    &= \int_{\cL^j}  \tilde{\psi^j}(y) -  \la y,\nabla \tilde{\psi^j}(y) \ra_{\cL^j}\d \bsigma^{\eta,\eps}_{j,(\underline t ,\underline x ),r}(y) + o_\eps(1)
    \\
    \fint_{C_j(\underline t ,\underline x ,r)} \partial_t f_{j,\eta,\eps} &= \int_{\cL^j}\H^j(\nabla \tilde{\psi^j}(y))\d \bsigma^{\eta,\eps}_{j,(\underline t ,\underline x ),r}(y) +o_\eta(1)\\
    \fint_{C_j(\underline t ,\underline x ,r)} \nabla f_{j,\eta,\eps} &= \int_{\cL^j}\nabla \tilde{\psi^j}(y)\d \bsigma^{\eta,\eps}_{j,(\underline t ,\underline x ),r}(y)\end{aligned}
\end{cases}
\end{align}
where the error term $o_\eps(1)$ is independent of $\eta,r$; the error term $o_\eta(1)$ is independent of $\eps,r$; and they satisfy $\lim_{\eps\to0}o_\eps(1)=0$ and $\lim_{\eta\to0}o_\eta(1)=0$.
\end{lemma}

In the above, $f_{j,\eta,\eps}=f_{j,\eta,\eps}(t,x)$ and $\fint_{C_j(\underline t ,\underline x ,r)} = \frac{1}{\Ll|C_j(\underline t ,\underline x ,r)\Rr|} \int_{C_j(\underline t ,\underline x ,r)} \d t \d x$.
The first identity was proved in \cite[Theorem~2.2~(ii)]{evans2014envelopes}.
The remaining two can be deduced in a similar way. For completeness, we present the proof below.

\begin{proof}
For simplicity, we identify $\cL^j$ with $\R^n$ for some $n\in\N$ and denote the inner product in~\eqref{e.<x,y>_cL^j=} by the dot product.
We write $f= f_{j,\eta,\eps}$, $\sigma^{s} = \sigma_{j,\eta,\eps,(s,\underline x), r}$ (solving~\eqref{e.sigma_PDE} with $(\underline t,\underline x)$ replaced by $ (s,\underline x)$), $\sF=\H^j_\eta$, and $\sF^i = \Ll(\partial_{x_i} \sF\Rr)(\nabla f)$ for $i\in\{1,\dots,n\}$. For any function $g$, we write $g_t = \partial_t g$ and $g_{x_i}= \partial_{x_i}g$ for $i\in\{1,\dots,n\}$. We also perform summation over $i$ without writing out the notation, for instance, $a_ib^i = \sum_{i=1}^na_ib^i$.
It is standard to see from~\eqref{e.sigma_PDE} that $\sigma^a(x)$ and $\nabla\sigma^a(x)$ decay exponentially fast as $x\to\infty$.
Hence, applications of integration by parts to be used below are valid.

Let us prove the first identity. We set
\begin{align}\label{e.w=f-xnablaf-tf_t}
    w = f-x\cdot \nabla f - t f_t,
\end{align}
and we can compute
\begin{align*}
    w_t = -x \cdot \nabla f_t - tf_{tt};\quad w_{x_i} = -x\cdot \nabla f_{x_i}- tf_{x_it};\quad \Delta w =-\Delta f - x\cdot \nabla \Delta f-t \Delta f_t.
\end{align*}
Hence,
\begin{align}\label{e.w_t+Fw=}
    w_t - \sF^i w_{x_i} = -x\cdot \Ll(\nabla f_t - \sF^i \nabla f_{x_i}\Rr) -t \Ll( f_{tt} - \sF^i f_{x_it}\Rr) = -\eps x\cdot \nabla \Delta f -\eps t\Delta f_t = \eps \Delta w+\eps \Delta f.
\end{align}
Let $s \in[\underline t,\underline t+r]$. At $t\in [0,s]$, we compute
\begin{align*}
    \frac{\d}{\d t}\int w \sigma^s \d x & = \int  w_t\sigma^s + w \sigma^s_t\d x
    \\
    &\stackrel{\eqref{e.w_t+Fw=}\eqref{e.sigma_PDE}}{=} \int \Ll(\sF^iw_{x_i} + \eps \Delta w + \eps \Delta f\Rr) \sigma^s + w\Ll(\Ll(\sigma^s \sF^i\Rr)_{x_i}-\eps \Delta \sigma^s\Rr)\d x
    \\
    & \stackrel{\eqref{e.integration_by_parts}}{=} \eps \int \Delta f \sigma^s \d x
\end{align*}
where the integration is over $\R^n=\cL^j$.
Integrating this in $t$ from $0$ to $s$ and using the terminal condition of $\sigma^s(s,\cdot)$ given in~\eqref{e.sigma_PDE}, we get
\begin{align*}
    \fint_{B_j(\underline x,r)} w(s,x) \d x = \int w(0,x)\sigma^s(0,x)\d x + \eps \int_0^s\int \Delta f(t,x) \sigma^s(t,x)\d x \d t.
\end{align*}
Notice $\underline t \leq s \leq \underline t +r \leq 2R$. Using \eqref{e.nabla^2f_bound} in Lemma~\ref{l.f,sigma_estimates} (with $s$ substituted for $\underline{t}$ therein), and applying Jensen's inequality (to the probability measure $\sigma^s \d x$; see in~\eqref{e.sigma^a>0}), we can bound the last term and get
\begin{align*}
    \fint_{B_j(\underline x,r)} w(s,x) \d x = \int w(0,x)\sigma^s(0,x)\d x + o_\eps(1),
\end{align*}
where $o_\eps(1)$ is independent of $\eta,r$.
Recall $C_j(\underline t,\underline x, r)$ in~\eqref{e.C()} and the probability measure $\bsigma^{\eta,\eps}_{j,(\underline t,\underline x), r}$ in~\eqref{e.sigma^eps_j,(t,x),r}.
Averaging the above over $s\in[\underline t,\underline t+r]$ and using~\eqref{e.w=f-xnablaf-tf_t}, we obtain the first identity in~\eqref{e.l.visc_app}.

Now, we turn to the second identity. Again let $s \in [\underline t,\underline t+r]$ and we compute for $t\in[0,s]$
\begin{align*}
    \frac{\d}{\d t}\int f_t \sigma^s \d x & = \int f_{tt}\sigma^s + f_t\sigma^s_t \d x 
    \\
    &\stackrel{\eqref{e.f_j,eta,eps=}\eqref{e.sigma_PDE}}{=} \int \Ll(\sF^i f_{tx_i}+\eps \Delta f_t\Rr) \sigma^s + f_t \Ll(\Ll(\sigma^s \sF^i\Rr)_{x_i} - \eps \Delta \sigma^s\Rr)\d x\stackrel{\eqref{e.integration_by_parts}}{=}0.
\end{align*}
For $\delta>0$ small, integrating the above display over $t \in [\delta,s]$ and using the terminal condition in~\eqref{e.sigma_PDE}, we get (recall $\sF= \H^j_\eta$)
\begin{align*}
    &\fint_{B_j(\underline x,r)} f_t(s,x) \d x = \int f_t(\delta,x)\sigma^s(\delta,x)\d x 
    \\
    &\stackrel{\eqref{e.f_j,eta,eps=}\eqref{e.lim|F_j-H^j|=0}}{=} \int \H^j_\eta\Ll(\nabla f(\delta,x)\Rr)\sigma^s(\delta,x)\d x  + \eps \int \Delta f(\delta,x) \sigma^s(\delta,x)\d x .
\end{align*}
Then, we send $\delta\to0$. 
Recall the bound on the difference between $\H^j_\eta$ and $\H^j$ in~\eqref{e.lim|F_j-H^j|=0}. Hence, replacing $\H^j_\eta$ by $\H^j$ gives rise to an error term $o_\eta(1)$ that vanishes as $\eta\to0$.
Notice that by Lemma~\ref{l.ext}, $\nabla f(0,\cdot) = \nabla\tilde{\psi^j}$ is Lipschitz. Therefore, $\Delta f(0,\cdot)$ is bounded almost everywhere by $C_j\eps$ for some constant $C_j$ depending on $j$. Therefore, we get
\begin{align*}
    &\fint_{B_j(\underline x,r)} f_t(s,x) \d x= \int \H^j\Ll(\nabla \tilde{\psi^j}(x)\Rr)\sigma^s(0,x)\d x  + o_\eta(1) 
\end{align*}
where $o_\eta(1)$ is independent of $\eps,r$.
Averaging the above over $s\in[\underline t,\underline t+r]$ and using~\eqref{e.C()} and~\eqref{e.sigma^eps_j,(t,x),r}, we obtain the second identity in~\eqref{e.l.visc_app}.

For the third identity, taking $s$ and $t$ as before, we compute for each $k\in\{1,\dots,n\}$
\begin{align*}
    &\frac{\d}{\d t}\int f_{x_k} \sigma^s \d x = \int f_{tx_k}\sigma^s + f_{x_k}\sigma^s_t \d x\\
    &\stackrel{\eqref{e.f_j,eta,eps=}}{=} \int \Ll(\sF^i f_{x_ix_k}+\eps \Delta f_{x_k}\Rr) \sigma^s + f_{x_k} \Ll(\Ll(\sigma^s \sF^i\Rr)_{x_i} - \eps \Delta \sigma^s\Rr)\d x\stackrel{\eqref{e.integration_by_parts}}{=}0.
\end{align*}
Proceeding similarly as above, we arrive at the third identity in~\eqref{e.l.visc_app}.
\end{proof}

\subsection{Other preliminary results}\label{s.other_prelim_results}

In this subsection, apart from the results above, we derive other results to be used in the next section.

\begin{lemma}[Approximation at a differentiable point]\label{l.touch_approx}
Let $j\in\N$ and let $f_j$ be given as in~\eqref{e.f_j=}.
Suppose that $f_j$ is differentiable at some $(t_j,x_j)\in(0,\infty)\times \itr(\mcl Q^j_2)$.
Then, there is a collection $(t_{j,\eta,\eps},x_{j,\eta,\eps})_{j\in\N,\eta>0,\eps>0}$ of points in $(0,\infty)\times \itr(\mcl Q^j_2)$ such that
\begin{gather}
    \lim_{\eta\to0}\limsup_{\eps\to0}\Ll|(t_{j,\eta,\eps},x_{j,\eta,\eps})- (t_j,x_j)\Rr|_{\R\times \cL^j}=0, \label{e.(t,x)->(t,mu)}
    \\
    \lim_{\eta\to0}\limsup_{\eps\to0}\Ll|(\partial_t, \nabla)f_{j,\eta,\eps}(t_{j,\eta,\eps},x_{j,\eta,\eps})-(\partial_t, \nabla)f_j(t_j,x_j)\Rr|_{\R\times\cL^j}=0.\label{e.|Df-Df|<delta}
\end{gather}
\end{lemma}

\begin{proof}
We write $D= (\partial_t ,\nabla)$ for brevity.
Since $f$ is differentiable at $(t_j,x_j)$, we can find a smooth increasing function $\omega:[0,\infty)\to[0,\infty)$ (for the existence of $\omega$, see the function $\gamma$ in \cite[Lemma~3.1.8]{cannarsa2004semiconcave}) satisfying $\omega(0)=\dot \omega(0)=0$ (here $\dot \omega$ is its derivative) such that, for all $(t,x)\in(0,\infty)\times\itr(\mcl Q^j_2)$,
\begin{align*}
    \Ll|f_j(t,x) - f_j(t_j,x_j) - \la (t,x)-(t_j,x_j),\, D f_j(t_j,x_j)\ra_{\R\times\cL^j}\Rr|\leq \omega\Ll(|(t,x) -(t_j,x_j)|_{\R\times\cL^j}\Rr).
\end{align*}
Define $\phi:[0,\infty)\times \mcl Q^j_2\to \R$ by
\begin{align*}
    \phi(t,x) = f_j(t,x) + \la (t,x) -(t_j,x_j), \, Df_j(t_j,x_j)\ra_{\R\times\cL^j} + 2\omega\Ll(|(t,x) -(t_j,x_j)|_{\R\times\cL^j}\Rr)
\end{align*}
for all $(t,x)\in[0,\infty)\times\mcl Q^j_2$,
which is differentiable everywhere. Notice that $\phi(t_j,x_j) = f_j(t_j,x_j)$.
We can always modify $\omega$ to satisfy $\omega(r)>0$ for $r>0$.
Since the above two displays yield
\begin{align}\label{e.f-phi<f-phi}
    f_j(t,x) - \phi(t,x)\leq f_j(t_j,x_j)-\phi(t_j,x_j) -\omega\Ll(|(t,x) -(t_j,x_j)|_{\R\times\cL^j}\Rr),
\end{align}
$f_j-\phi$ achieves a strict maximum at $(t_j,x_j)$.
Fix any $r\in(0, t/2)$ to ensure that the set defined by
\begin{align*}
    B=\Ll\{(t,x)\in(0,\infty)\times \C^j_2:\:|(t,x)-(t_j,x_j)|_{\R\times\cL^j}\leq r\Rr\}
\end{align*}
satisfies $B\subset (0,\infty)\times \itr(\mcl Q^j_2)$.
Let $(t_{j,\eta,\eps},x_{j,\eta,\eps})\in B$ be a point at which $f_{j,\eta,\eps}-\phi$ achieves the maximum over $B$.

Let us fix an arbitrary $\delta\in(0,r)$.
The convergence of $f_{j,\eta} $ to $f_j$ (Lemmas~\ref{l.f_j,0_cvg_f_j} and~\ref{l.f_j,eta_cvg_f_j,0}) and that of $f_{j,\eta,\eps}$ to $f_{j,\eta}$ (Lemma~\ref{l.f_j,eta,eps_to_f_j,eta}) allow us to choose $c>0$ and $(c_\eta)_{\eta>0}$ such that
\begin{align}
    \sup_{\eta<c}\sup_{\eps<c_\eta}\sup_{B}\Ll|f_{j,\eta,\eps}-f_j\Rr|< \frac{\omega(\delta)}{4}.\label{e.sup_B|f-f|}
\end{align}
To show~\eqref{e.(t,x)->(t,mu)}, we argue by contradiction and suppose 
\begin{align}\label{e.contrad_assmp}
    |(t_{j,\eta,\eps},x_{j,\eta,\eps})-(t_j,x_j)|_{\R\times\cL^j}\geq\delta
\end{align}
for some $\eta<c$ and $\eps<c_\eta$.
Then, we have
\begin{align*}
    & f_{j,\eta,\eps}(t_{j,\eta,\eps},x_{j,\eta,\eps}) - \phi(t_{j,\eta,\eps},x_{j,\eta,\eps}) 
    \stackrel{\eqref{e.sup_B|f-f|}}{\leq} f_j(t_{j,\eta,\eps}, x_{j,\eta,\eps})-\phi(t_{j,\eta,\eps},x_{j,\eta,\eps})+\frac{\omega(\delta)}{4}
    \\
    &\stackrel{\eqref{e.f-phi<f-phi}\eqref{e.contrad_assmp}}{\leq} f_j(t_j,x_j)-\phi(t_j,x_j) -\frac{3\omega(\delta)}{4}
     \stackrel{\eqref{e.sup_B|f-f|}}{\leq} f_{j,\eta,\eps}(t_j,x_j)-\phi(t_j,x_j) - \frac{\omega(\delta)}{2}
\end{align*}
which contradicts the fact that $f_{j,\eta,\eps} - \phi$ achieves the maximum over $B$ at $(t_{j,\eta,\eps},x_{j,\eta,\eps})$. 
Therefore, by contradiction, we must have $|(t_{j,\eta,\eps},x_{j,\eta,\eps})-(t_j,x_j)|_{\R\times\cL^j}<\delta$ for every $\eps<c_\eta$ and $\eta<c$. Since $\delta$ is arbitrary, we get~\eqref{e.(t,x)->(t,mu)}

Lastly, we show~\eqref{e.|Df-Df|<delta}. By~\eqref{e.(t,x)->(t,mu)}, we have $|(t_{j,\eta,\eps},x_{j,\eta,\eps})- (t,\pj_j{q})|_{\R\times\cL^j}<r$ for sufficiently small $\eta,\eps>0$. Then, $f_{j,\eta,\eps}-\phi$ achieves its maximum over $B$ in the interior of $B$, which implies $D f_{j,\eta,\eps}(t_{j,\eta,\eps},x_{j,\eta,\eps}) = D\phi(t_{j,\eta,\eps},x_{j,\eta,\eps})$. 
On the other hand, the definition of $\phi$ gives $D f_j(t_j,x_j) = D \phi(t_j,x_j)$. 
These along with~\eqref{e.(t,x)->(t,mu)} and the smoothness of $\phi$ imply~\eqref{e.|Df-Df|<delta}.
\end{proof}

Recall the probability measure $\bsigma^{\eta,\eps}_{j,(t,x),r}$ introduced in~\eqref{e.sigma^eps_j,(t,x),r}.
\begin{lemma}[First moment bound]\label{l.1stmomet}

\rv{For every} $j,\eta, \eps,t,x,r$, we have
\begin{align*}
    \int_{\cL^j}|x'|_{\cL^j} \d\bsigma^{\eta,\eps}_{j,(t,x),r}(x') \leq (1+\|\H\|_\mathrm{Lip}+\eps)\Ll(1+|x|_{\cL^j}+t+r\Rr).
\end{align*}
\end{lemma}

\begin{proof}
We set ${p}(x') = \sqrt{1+|x'|^2_{\cL^j}}$ for every $x'$ and use the shorthand notation $\sigma_s =\sigma_{j,\eta,\eps,(s,x),r}$ for $s\in[t,t+r]$. 
Using the equation~\eqref{e.sigma_PDE} satisfied by $\sigma_s$ and integration by parts in~\eqref{e.integration_by_parts}, we can compute
\begin{align*}
    \frac{\d}{\d t'}\int_{\cL^j}{p} (x')\sigma_s(t',x')\d x'  = -\int_{\cL^j}\Ll( \la\nabla{p},\nabla\H^j_\eta(\nabla f_{j,\eta,\eps})\ra_{\cL^j}+\eps \Delta {p}\Rr)\sigma_s  \geq -\Ll(\|\H\|_\mathrm{Lip}+\eps\Rr),
\end{align*}
where we used $|\nabla{p}|_{\cL^j}\leq 1$, $\Delta {p}\leq 1$ (due to the weight $\frac{1}{2^j}$ in $\Delta$, see~\eqref{e.Delta,div=}), the uniform bound $\|\H\|_\mathrm{Lip}$ on $\nabla\H^j_\eta$ in~\eqref{e.bound_on_nabla_F_j}, and $\int_{\cL^j}\sigma_s=1$ in~\eqref{e.sigma^a>0}. Integrating this in $t'$ from $0$ to $s$, we get
\begin{align}
    \int_{\cL^j}{p}(x')\sigma_s(0,x')\d x' & \leq \int_{\cL^j}{p}(x')\sigma_s(s,x')\d x' + (\|\H\|_\mathrm{Lip}+\eps)s \notag
    \\
    & \stackrel{\eqref{e.sigma_PDE}}{=} \frac{1}{|B_j(x,r)|_{\cL^j}}\int_{B_j(x,r)} \sqrt{1+|x'|^2_{\cL^j}}\d x' + (\|\H\|_\mathrm{Lip}+\eps)s \notag
    \\
    & \stackrel{\eqref{e.B_j(x,r)=}}{\leq} 1+|x|_{\cL^j}+r+(\|\H\|_\mathrm{Lip}+\eps)s. \label{e.<1+|x|+r+...}
\end{align}
Then, we can get
\begin{align*}
    \int_{\cL^j}|x'|_{\cL^j} \d\bsigma^{\eta,\eps}_{j,(t,x),r}(x') &\leq \int_{\cL^j}{p}(x') \d\bsigma^{\eta,\eps}_{j,(t,x),r}(x') \stackrel{\eqref{e.sigma^eps_j,(t,x),r}}{=} \fint_t^{t+r}\int_{\cL^j}{p}(x')\sigma_{s}(0,x')\d x' \d s
    \\
    &\stackrel{\eqref{e.<1+|x|+r+...}}{\leq} 1+|x|_{\cL^j} + r + (\|\H\|_\mathrm{Lip}+\eps)(t+r)
\end{align*}
which gives the announced bound.
\end{proof}

Note that the same argument in deriving~\eqref{e.<1+|x|+r+...} yields
\begin{align}\label{e.int_|x|<}
    \int_{\cL^j}|x'|_{\cL^j}\sigma_{j,\eta,\eps,(s,x),r}(t',x')\d x' \leq 1 + |x|_{\cL^j} + r +(\|\H\|_\mathrm{Lip}+\eps)(t+r)
\end{align}
for all $t'\in [0,s]$ and $s\in[t,t+r]$, which will be used below.

\begin{lemma}[Second moment bound]\label{l.2ndmomet}

There is a constant $C>0$ such that
\begin{align*}
    \int_{\cL^j}|x'|^2_{\cL^j} \d\bsigma^{\eta,\eps}_{j,(t,x),r}(x') \leq (C+\eps)\Ll(C+|x|^2_{\cL^j}+t^2+r^2\Rr)
\end{align*}
uniformly in $j,\eta,\eps,t,x,r$.
\end{lemma}

\begin{proof}
We set ${p}(x') =|x'|^2_{\cL^j}$ for every $x'$ and again use the shorthand notation $\sigma_s =\sigma_{j,\eta,\eps,(s,x),r}$ for $s\in[t,t+r]$.
Using the equation~\eqref{e.sigma_PDE} satisfied by $\sigma_s$ and integration by parts in~\eqref{e.integration_by_parts}, we can compute
\begin{align*}
    \frac{\d}{\d t'}\int_{\cL^j}{p} (x')\sigma_s(t',x')\d x'  = -\int_{\cL^j}\Ll( \la\nabla{p},\nabla\H^j_\eta(\nabla f_{j,\eta,\eps})\ra_{\cL^j}+\eps \Delta {p}\Rr)\sigma_s 
    \\
    \geq -(C+\eps)(C+|x|_{\cL^j}+t+r),
\end{align*}
for some absolute constant $C>0$,
where we used $|\nabla{p}(x')|_{\cL^j}= 2|x'|_{\cL^j}$, $\Delta {p} = 2$ (due to the weight $\frac{1}{2^j}$ in $\Delta$, see~\eqref{e.Delta,div=}), the boundedness of $\nabla\H^j_\eta$ in~\eqref{e.bound_on_nabla_F_j}, and~\eqref{e.int_|x|<}. Integrating this in $t'$ from $0$ to $s$, we get
\begin{align*}
    \int_{\cL^j}{p}(x')\sigma_s(0,x')\d x' & \leq \int_{\cL^j}{p}(x')\sigma_s(s,x')\d x' + (C+\eps)(C+|x|_{\cL^j}+t+r)s
    \\
    & \stackrel{\eqref{e.sigma_PDE}}{=} \frac{1}{|B_j(x,r)|_{\cL^j}}\int_{B_j(x,r)} |x'|^2_{\cL^j}\d x' + (C+\eps)(C+|x|_{\cL^j}+t+r)s
    \\
    & \stackrel{\eqref{e.B_j(x,r)=}}{\leq} 2\Ll(|x|^2_{\cL^j}+r^2\Rr)+(C+\eps)(C+|x|_{\cL^j}+t+r)s.
\end{align*}
Averaging this in $s$ from $t$ to $t+r$ and using the above estimate, we get
\begin{align*}
    \int_{\cL^j}|x'|^2_{\cL^j} \d\bsigma^{\eta,\eps}_{j,(t,x),r}(x') & \stackrel{\eqref{e.sigma^eps_j,(t,x),r}}{=} \fint_t^{t+r}\int_{\cL^j}{p}(x')\sigma_{s}(0,x')\d x' \d s
    \\
    &\leq (C+\eps)\Ll(C+|x|^2_{\cL^j}+t^2+r^2\Rr)
\end{align*}
for some larger $C$, as announced.
\end{proof}
One can continue iterating this to obtain bounds on higher moments. For our purpose, bounds on the first two moments are sufficient.

We equip $L^2$ with the Borel sigma-algebra with respect to the $L^2$-metric and consider probability measures on it. 
Recall the lift map $\lf_j:\cL^j\to L^2$ from~\eqref{e.pj,lj=}.
We need the following tightness criterion.

\begin{lemma}[Tightness]\label{l.weak_cvg}
Let $(\gamma_j)_{j\in \N}$ be a sequence of Borel probability measures on $\cH$. Assume $\gamma_j = \bsigma_j\circ \lf_j^{-1}$ for some Borel probability measure $\bsigma_j$ on $\cL^j$ for each $j$ and assume $\sup_{j}\int_\cH |q|_{L^2}\d \gamma_j(q)<\infty$. Then, there is a subsequence $(\gamma_{j_n})_{n\in\N}$ converging weakly to some probability measure $\gamma_\infty$ on $\cH$.

\end{lemma}
The weak convergence here is the standard one: $\lim_{n\to\infty}\int_\cH g\,\d \gamma_{j_n} = \int_\cH g\,\d \gamma_\infty$ for every bounded continuous function $g:\cH\to\R$.

\begin{proof}
For each $k\in\N$, the bound $\sup_{j}\int_\cH |q|_{L^2}\d \gamma_j(q)<\infty$ implies that, for every $k$, the sequence $(\gamma_j\circ \pj_k^{-1})_{j\in\N}$ of probability measures on $\cL^k$ is tight. Hence, we can pass to a weakly converging subsequence. 
Let $(j^1_n)_{n=1}^\infty$ be an increasing sequence of positive integers along which $\gamma_{j^1_n}\circ\pj_1^{-1}$ converges to some probability measure $\varsigma_1$ on $\cL^1$ as $n\to\infty$. Then, inductively, given a sequence $(j^k_n)_{n=1}^\infty$, we extract from it a further subsequence $(j^{k+1}_n)_{n=1}^\infty$ such that $\gamma_{j^{k+1}_n}\circ\pj_{k+1}^{-1}$ converges to some probability measure $\varsigma_{k+1}$ on $\cL^{k+1}$ as $n\to\infty$. 

Setting $\tilde \gamma_n=
\gamma_{j^n_n}$ for each $n\in\N$, we have that for every $k\in\N$, the sequence $(\tilde\gamma_n\circ\pj_k^{-1})_{n\in\N}$ of probability measures on $\cL^k$ converges weakly to $\varsigma_k$ as $n\to\infty$. 

Using this convergence and Lemma~\ref{l.basics_(j)}~\eqref{i.projective} (the projective property), we can deduce $\varsigma_{k'}\circ(\pj_k\lf_{k'})^{-1} = \varsigma_k$ for $k'\geq k$, which implies that $(\varsigma_k)_{k\in\N}$ is consistent with respect to projections $\pj_k \lf_{k'} :\cL^{k'}\to \cL^k$ indexed by $k'\geq k$. Kolmogorov's extension theorem (c.f.\ \cite[Theorem~2.4.3]{tao2011introduction}) yields a probability measure $\gamma_\infty$ on $\cH$ such that $\gamma_\infty\circ\pj_k^{-1} = \varsigma_k$ for every $k$.

Using \cite[Theorem~3.4]{ccapar1993weak} (more precisely, we use the equivalence between i) and ii) therein; we substitute $(\bsigma_{j^n_n})_n$ and $\gamma_\infty$ for $(\mu_\alpha)_\alpha$ and $\mu$ in its statement), we can upgrade the weak convergence of $(\tilde\gamma_n\circ\pj_k^{-1})_{n\in\N}$ for every $k$ to the weak convergence of $\tilde \gamma_n$. Choosing $(j_n)_{n\in\N}$ via $\gamma_{j_n}=\tilde \gamma_n$, we obtain the desired result.
\end{proof}

Lastly, we show that asymptotically the measure $\bsigma^{\eta,\eps}_{j,(t,x),r}$ in~\eqref{e.sigma^eps_j,(t,x),r} is supported on the cone $\mcl Q_2^j$ if $x$ is away from the boundary. 

\begin{lemma}[Support on the cone]\label{l.supp_cone}
Let $j\in\N$, $R>0$, and $\underline x\in \itr(\mcl Q^j_2)$. Then, we have
\begin{align*}
    \lim_{c\to0}\sup_{\substack{\eta,\eps,r\in(0,c)\\ |(t,x)|_{\R\times \cL^j}\leq R\\ x\in \underline x+\C_2^j}}\bsigma^{\eta,\eps}_{j,(t,x),r}\Ll(\cL^j\setminus\C_2^j\Rr)=0.
\end{align*}
\end{lemma}

\begin{proof}
It suffices to show that $\bsigma^{\eta,\eps}_{j,(t_{\eta,\eps,r},x_{\eta,\eps,r}),r}\Ll(\cL^j\setminus\C_2^j\Rr)$ converges to $0$ along any vanishing subsequence of $(\eta,\eps,r)$ along which $(t_{\eta,\eps,r},x_{\eta,\eps,r})$ satisfies the constraint and $(\bsigma^{\eta,\eps}_{j,(t_{\eta,\eps,r},x_{\eta,\eps,r}),r})_{\eta,\eps,r>0}$ converges weakly to some probability measure $\gamma$. (Indeed, supposing that the lemma does not hold, we can use this and the tightness of $(\bsigma^{\eta,\eps}_{j,(t_{\eta,\eps,r},x_{\eta,\eps,r}),r})_{\eta,\eps,r}$ guaranteed by Lemmas~\ref{l.1stmomet} to produce a counterexample.) For brevity, we avoid the notation of subsequences and assume that $\bsigma^{\eta,\eps}_{j,(t_{\eta,\eps,r},x_{\eta,\eps,r}),r}$ converges to $\gamma$ as $\eta,\eps,r\to 0$.

First, we want to show that $\gamma(\underline x+\C_2^j)=1$. Let $\phi:\R\to[0,1]$ be any smooth function with bounded derivatives and satisfying $\phi =0 $ on $(-\infty,0]$, $\phi>0$ on $(0,\infty)$, and the derivative $\phi'\geq 0$. We write $\|\phi'\|_\infty = \sup |\phi'|$ and similarly for $\|\phi''\|_\infty$.
Fix any $a \in (\C_2^j)^*$ (see the notion of dual cones in Definition~\ref{d.dual_cone}) and set ${p}(x') = \phi( - \la a, x'-\underline x\ra_{\cL^j})$ for every $x'$. Then, we can compute that $\nabla {p} = - \phi' a $ and $\Delta {p} = \phi''|a|^2_{\cL^j}$. 
Recall from~\ref{i.d2} (below~\eqref{e.H=}) that $\H$ is $\mcl Q^*_2$-increasing (see Definition~\ref{d.dual_cone}), which along with Lemma~\ref{l.proj_cones} (the second inclusion therein) implies that $\H^j$ (defined as in~\eqref{e.j-projection}) is $(\mcl Q^j_2)^*$-increasing. 
From the definition of $\H^j_\eta$ from~\eqref{e.H^j_eta=}, it is clear that $\H^j_\eta$ is also $(\mcl Q^j_2)^*$-increasing. By Definition~\ref{d.dual_cone} of dual cones, this implies that $\nabla \H^j_\eta$ takes value in the dual of $(\mcl Q_2^j)^*$, which is equal to $\mcl Q_2^j$ by Lemma~\ref{l.duality_cones}. Hence, we have $\la a, \nabla \H^j_\eta\ra_{\cL^j}\geq 0$.
Henceforth, we write $\sigma_s = \sigma_{j,\eta,\eps,(s,x_{\eta,\eps,r}),r}$ for $s\in [t_{\eta,\eps,r},t_{\eta,\eps,r}+r]$.
Using these, the equation satisfied by $\sigma_s$ in~\eqref{e.sigma_PDE}, and integration by parts in~\eqref{e.integration_by_parts}, we have
\begin{align*}
    \frac{\d}{\d t'}\int_{\cL^j}{p} \sigma_s(t',\cdot)  = -\int_{\cL^j}\Ll( \la\nabla{p},\nabla\H^j_\eta(\nabla f_{j,\eta,\eps})\ra_{\cL^j}+\eps \Delta {p}\Rr) \sigma_s(t',\cdot) \geq - C\eps
\end{align*}
where $C= \|\phi''\|_\infty |a|^2_{\cL^j}$. Integrating this in $t'$ from $0$ to $s$, we get
\begin{align*}
    \int_{\cL^j}{p}(x')\sigma_s(0,x')\d x' & \leq \int_{\cL^j}{p}(x')\sigma_s(s,x')\d x' + C\eps s
    \\
    & \stackrel{\eqref{e.sigma_PDE}}{=} \frac{1}{|B_j(x_{\eta,\eps,r},r)|}\int_{B_j(x_{\eta,\eps,r},r)} \phi\Ll( - \la a, x'-\underline x\ra_{\cL^j}\Rr)\d x' + C\eps s.
\end{align*}
Due to $x_{\eta,\eps,r}\in \underline x+\C_2^j$, we have $\la a, x_{\eta,\eps,r}-\underline x \ra_{\cL^j}\geq 0$. Hence, since $\phi$ is increasing and $\phi(0)=0$, we have
\begin{align*}
    \int_{B_j(x_{\eta,\eps,r},r)} \phi\Ll( - \la a, x'-\underline x\ra_{\cL^j}\Rr)\d x' 
    & \stackrel{\eqref{e.B_j(x,r)=}}{=} \int_{B_j(0,r)} \phi\Ll( - \la a, x'+x_{\eta,\eps,r}-\underline x\ra_{\cL^j}\Rr)\d x' 
    \\
    &\leq \int_{B_j(0,r)} \phi\Ll( - \la a, x'\ra_{\cL^j}\Rr)\d x'
    \stackrel{\eqref{e.B_j(x,r)=}}{\leq} |B_j(0,r)|\|\phi'\|_\infty |a|_{\cL^j}r.
\end{align*}
Therefore,
\begin{align*}
    \int_{\cL^j}{p}(x')\sigma_s(0,x')\d x'  \leq  \|\phi'\|_\infty |a|_{\cL^j}r+C\eps s.
\end{align*}
Averaging this in $s$ from $t_{\eta,\eps,r}$ to $t_{\eta,\eps,r}+r$ and using the definition of $\bsigma^{\eta,\eps}_{j,(t,x),r}$ in~\eqref{e.sigma^eps_j,(t,x),r} we get
\begin{align*}
    \int_{\cL^j}{p} (x') \d\bsigma^{\eta,\eps}_{j,(t_{\eta,\eps,r},x_{\eta,\eps,r}),r}(x') \leq C(r+\eps+r\eps)
\end{align*}
for some $C$ independent of $r$ and $\eps$ (as long as $r,\eps<1$).
Sending $\eta,\eps,r$ to $0$, we get
\begin{align*}
    \int_{\cL^j} \phi\Ll( - \la a, x'-\underline x\ra_{\cL^j}\Rr) \d \gamma(x') =0.
\end{align*}
Since $\phi$ is arbitrary, $\gamma$ is supported on $S_a = \{x':\la a, x'-\underline x\ra_{\cL^j}\geq 0\}$. Taking $\cup_a S_a$ over a countable dense set in $(\C_2^j)^*$, we have that $\gamma$ is supported on $\{x':\la a, x'-\underline x\ra_{\cL^j}\geq 0,\ \forall a \in (\C_2^j)^*\}$ which is exactly $\underline x+\C_2^j$ by Lemma~\ref{l.duality_cones}. 

Due to $\underline x\in \itr(\C_2^j)$, we have $\underline x+\C_2^j\subset \itr(\C_2^j)$. By the Portmanteau theorem, we have
\begin{align*}
    \liminf_{\eta,\eps,r\to0} \bsigma^{\eta,\eps}_{j,(t_{\eta,\eps,r},x_{\eta,\eps,r}),r}\Ll(\itr(\C_2^j)\Rr) \geq \gamma\Ll(\itr(\C_2^j)\Rr)\geq \gamma\Ll(\underline x+\C_2^j\Rr)=1,
\end{align*}
which implies $\lim_{\eta,\eps,r\to0} \bsigma^{\eta,\eps}_{j,(t_{\eta,\eps,r},x_{\eta,\eps,r}),r}(\C_2^j)=1$. This along with the discussion in the beginning completes the proof.
\end{proof}

\section{Proofs of main results}\label{s.proofs}

In this section, we prove Theorems~\ref{t} and~\ref{t2}. First, we combine results from previous sections to prove Proposition~\ref{p.2nd_eq_diff_pt} which resembles Theorem~\ref{t2} but involves approximations $f_{j,\eta,\eps}$ (see~\eqref{e.f_j,eta,eps=}) instead of $f_j$ (see~\eqref{e.f_j=}). Then, we use Proposition~\ref{p.2nd_eq_diff_pt} to prove Theorem~\ref{t2} and then deduce Theorem~\ref{t} from Theorem~\ref{t2}.

\begin{proposition}\label{p.2nd_eq_diff_pt}
We assume the setup in Theorem~\ref{t}. Let $(t,q) \in [0,\infty)\times \mcl Q_2$, let $\seq$ be a strictly increasing subsequence of $\N$, and let $(t_{j,\eta,\eps}, x_{j,\eta,\eps})_{j\in \seq,\eta>0,\eps>0}$ be a sequence satisfying the following assumptions:
\begin{enumerate}[label=(\roman*)]
    \item \label{i.p_i} for every $j\in\seq$, there is $\underline{x}_j \in \itr\Ll(\mcl Q^j_2\Rr)$ such that $(t_{j,\eta,\eps},x_{j,\eta,\eps}) \in (0,\infty) \times \Ll(\underline{x}_j + \mcl Q^j_2\Rr)$ for sufficiently small $\eta,\eps>0$;
    \item\label{i.p_ii} we have
    \begin{align*}
        \lim_{j\in\seq} \limsup_{\eta\to0}\limsup_{\eps\to0} \Ll|\Ll(t_{j,\eta,\eps}, \lf_j x_{j,\eta,\eps}\Rr) - (t, q)\Rr|_{\R\times L^2} =0;
    \end{align*}
    \item\label{i.p_iii} there is $(a,p)\in\R\times L^2$ such that
    \begin{align*}
        \lim_{j\in\seq} \limsup_{\eta\to0}\limsup_{\eps\to0} \Ll|\Ll(\partial_t ,\bnabla\Rr) f^\uparrow_{j,\eta,\eps}\Ll(t_{j,\eta,\eps}, \lf_j x_{j,\eta,\eps}\Rr) - (a,p)\Rr|_{\R\times L^2} =0.
    \end{align*}
\end{enumerate}
Then, there is a Borel probability measure $\gamma_{t,{q}}$ on $\C_2$ such that
\begin{gather}\label{e.p}
\begin{split}
    f(t,{q}) - \la {q}, p\ra_\cH -ta = \int_{\C_2} \psi({q'}) -\la{q'}, \bnabla\psi({q'})\ra_{\cH}  \d \gamma_{t,{q}}({q'}),
    \\
    a = \int_{\C_2} \bxi\left(\bnabla\psi({q'})\right)\d \gamma_{t,{q}}({q'}),\qquad p = \int_{\C_2} \bnabla\psi({q'}) \d \gamma_{t,{q}}({q'}).
\end{split}
\end{gather}
\end{proposition}

One can interpret $\underline{x}_j$ in \ref{i.p_i} as a uniform lower bound on $x_{j,\eta,\eps}$ in the partial order induced by $\mcl Q^j_2$ so that $x_{j,\eta,\eps}$ stays strictly positive uniformly in $\eta$ and $\eps$.
In~\ref{i.p_iii}, $f_{j,\eta,\eps}$ is given in~\eqref{e.f_j,eta,eps=} and its lift $f^\uparrow_{j,\eta,\eps}$ is defined as in~\eqref{e.lift=}.

\begin{proof}
Since no properties on $\seq$ are required for this proof except for the strict monotonicity, we assume $\seq =\N$ for simplicity.
In \rv{the} following, if not specified, $\eta,\eps,r$ are taken in $(0,1)$.

\textbf{Step~0.}
We introduce notation and derive some basic estimates. The key estimate is~\eqref{e.L-R=0}.
For $\eta,\eps,r\in(0,1)$, we write
\begin{gather*}\bsigma_{j,\eta,\eps,r} =\bsigma^{\eta,\eps}_{j,(t_{j,\eta,\eps},x_{j,\eta,\eps}),r},\qquad  C_{j,\eta,\eps,r} = C_j(t_{j,\eta,\eps},x_{j,\eta,\eps},r)
\end{gather*}
which are defined in~\eqref{e.sigma^eps_j,(t,x),r} and~\eqref{e.C()}, respectively.
Applying Lemma~\ref{l.visc_app} to each fixed $j$ (with $R$ therein taken to satisfy $R>1$ and $R>t_{j,\eta,\eps}$ for all sufficiently small $\eta,\eps$) by substituting $t_{j,\eta,\eps},x_{j,\eta,\eps},t',x'$ for $\underline t,\underline x,t,x$ therein, we can get
\begin{align}\label{e.evans}
\begin{cases}
    \begin{aligned}
        &\fint_{C_{j,\eta,\eps,r}} f_{j,\eta,\eps} - \la x', \nabla f_{j,\eta,\eps}\ra_{\cL^j} -t' \partial_tf_{j,\eta,\eps} 
        \\
        &\qquad\qquad\qquad= \int_{\cL^j} \tilde{\psi^j}(y) - \la y,\nabla\tilde{\psi^j}(y)\ra_{\cL^j} \d \bsigma_{j,\eta,\eps,r}(y)
        +o^j_\eps(1),
    \\
    &\fint_{C_{j,\eta,\eps,r}} \partial_t f_{j,\eta,\eps} = \int_{\cL^j}\H^j(\nabla\tilde{\psi^j}(y))\d \bsigma_{j,\eta,\eps,r}(y) + o^j_\eta(1),
\\
    &\fint_{C_{j,\eta,\eps,r}} \nabla f_{j,\eta,\eps} = \int_{\cL^j}\nabla\tilde{\psi^j}(y)\d \bsigma_{j,\eta,\eps,r}(y) ,\end{aligned}
\end{cases}
\end{align}
where $f_{j,\eta,\eps}=f_{j,\eta,\eps}(t',x')$ and $\fint_{C_{j,\eta,\eps,r}}$ is over $(t',x')$. The error \rv{term} $ o^j_\eps(1)$ is independent of $\eta,r$ and $ o^j_\eta(1)$ is independent of $\eps,r$.  They satisfy $\lim_{\eps\to0} o^j_\eps(1)=0$ and $\lim_{\eta\to0} o^j_\eta(1)=0$.
Then, we introduce shorthand notation for integrands in \eqref{e.evans}:
\begin{gather*}
    \mathsf{L}^1_{j,\eta,\eps}(s,{q'}) = f^\uparrow_{j,\eta,\eps}(s,{q'}) - \la {q'}^\j, \bnabla f_{j,\eta,\eps}^\uparrow(s,{q'})\ra_{\cH} -t\partial_tf^\uparrow_{j,\eta,\eps}(s,{q'}),
    \\
    \mathsf{L}^2_{j,\eta,\eps}(s,{q'}) = \partial_tf^\uparrow_{j,\eta,\eps}(t,{q'}),\qquad \mathsf{L}^3_{j,\eta,\eps}(s,{q'}) = \bnabla f_{j,\eta,\eps}^\uparrow(t,{q'}),
    \\
    \tilde{\mathsf{R}}^1_{j}({q'}) = \tilde{\psi^j}\Ll(\pj_j{q'}\Rr) - \la {q'},\lf_j\nabla\tilde{\psi^j}\Ll(\pj_j{q'}\Rr)\ra_{\cH},
    \\
    \tilde{\mathsf{R}}^2_{j}({q'}) = \H\Ll(\lf_j\nabla \tilde{\psi^j}\Ll(\pj_j{q'}\Rr)\Rr),\qquad \tilde{\mathsf{R}}^3_{j}({q'}) = \lf_j\nabla \tilde{\psi^j}\Ll(\pj_j{q'}\Rr).
\end{gather*}
Here, $\mathsf{L}^i_{j,\eta,\eps}$ and $\tilde{\mathsf{R}}^i_{j}$ corresponds, respectively, to the left- and (the main term in) the right-hand side of the $i$-th equation in~\eqref{e.evans}.
Notice that, by~\eqref{e.kappa^(j)} and Lemma~\ref{l.basics_(j)}~\eqref{i.pj_lf=ID}, we have
\begin{align}\label{e.(lx)^j=lx}
    \Ll(\lf_j x'\Rr)^\j = \lf_j x',\quad\forall x'\in \cL^j; \qquad \pj_j\lf_j y  = y,\quad\forall y \in \cL^j.
\end{align}
By Lemma~\ref{l.basics_(j)}~\eqref{i.isometric} and Lemma~\ref{l.derivatives}~\eqref{i.lift_gradient}, we have
\begin{align}\label{e.<>_L^2=<>_L^j}
\begin{split}
    \la \lf_j x', \bnabla f^\uparrow_{j,\eta,\eps}(s,\lf_jx')\ra_{L^2}= \la x', \nabla f_{j,\eta,\eps}(s, x')\ra_{\cL^j}',\quad&\forall (s,x')\in \R\times\cL^j;
    \\
    \la \lf_j y,\, \lf_j \nabla \tilde{\psi^j}(\pj_j\lf_j y)\ra_{L^2} = \la y,  \nabla \tilde{\psi^j}( y)\ra_{\cL^j},\quad&\forall y\in \cL^j.
\end{split}
\end{align}
We write
\begin{align*}
    \text{$|\cdot|_i = |\cdot|_\R$ for $i\in\{1,2\}$ \qquad and \qquad $|\cdot|_3 = |\cdot|_\cH$}.
\end{align*}
Then, we can rewrite~\eqref{e.evans} as
\begin{align*}\Ll|\fint_{C_{j,\eta,\eps,r}}\mathsf{L}^i_{j,\eta,\eps}(\cdot, \lf_j\cdot) - \int_{\cL^j}\tilde{\mathsf{R}}^i_j(\lf_j\cdot) \d \bsigma_{j,\eta,\eps,r} \Rr|_i \leq o^j_\eps(1)+o^j_\eta(1) ,\quad\forall i\in\{1,2,3\}.
\end{align*}
Let us explain this. The relation for $i=1$ follows from~\eqref{e.(lx)^j=lx}, \eqref{e.<>_L^2=<>_L^j}, and the definition of the lift of functions in~\eqref{e.lift=}. 
The relation for $i=2$ follows from~\eqref{e.(lx)^j=lx} and the definition of the projection of functions in~\eqref{e.j-projection}.
The relation for $i=3$ follows from Lemma~\ref{l.derivatives}~\eqref{i.lift_gradient} and~\eqref{e.(lx)^j=lx}.
This allows us to choose a vanishing sequence $(\eta_j,\eps_j,r_j)_{j\in\N}$ such that
\begin{align}\label{e.L-R=0}
    \lim_{j\to\infty}\sum_{i=1}^3\Ll|\fint_{C_{j,\eta_j,\eps_j,r_j}}\mathsf{L}^i_{j,\eta_j,\eps_j}(\cdot, \lf_j\cdot) - \int_{\cL^j}\tilde{\mathsf{R}}^i_j(\lf_j\cdot) \d \bsigma_{j,\eta_j,\eps_j,r_j} \Rr|_i=0.
\end{align}
Later, if necessary, we will choose smaller values for $\eta_j,\eps_j,r_j$.
We want to compare the terms in~\eqref{e.L-R=0} with the terms in~\eqref{e.p}. We start with the left-hand sides.

\textbf{Step~1.} 
We want to show
\begin{align}\label{e.limsupL}
    \lim_{j\to\infty}\sum_{i=1}^3\Ll|\fint_{C_{j,\eta_j,\eps_j,r_j}}\mathsf{L}^i_{j,\eta_j,\eps_j}(t',\lf_j x') -\mathsf{L}^i_\infty\Rr|_i=0,
\end{align}
where 
\begin{align}\label{e.L_infty=}
    \mathsf{L}^1_\infty = f(t,{q}) - \la {q}, p\ra_\cH -ta,\quad\mathsf{L}^2_\infty = a,\quad \mathsf{L}^3_\infty =p
\end{align}
correspond to terms on the left-hand sides in~\eqref{e.p}.
Using assumptions~\ref{i.p_ii} and~\ref{i.p_iii}, the local uniform convergence of $f^\uparrow_{j,\eta,\eps}$ to $f$ (see Lemmas~\ref{l.f_j_approx}, \ref{l.f_j,0_cvg_f_j}, \ref{l.f_j,eta_cvg_f_j,0}, \ref{l.f_j,eta,eps_to_f_j,eta}) and~\eqref{e.(lx)^j=lx}, we can choose smaller $\eta_j,\eps_j$ so that
\begin{align}
    \lim_{j\to\infty}\sum_{i=1}^3\Ll|\mathsf{L}^i_{j,\eta_j,\eps_j}(t_{j,\eta_j,\eps_j},\lf_j x_{j,\eta_j,\eps_j}) - \mathsf{L}^i_\infty\Rr|_i=0 .\label{e.limsupL^1}
\end{align}
Then, using the continuity of $\mathsf{L}^i_{j,\eta_j,\eps_j}(\cdot,\lf_j\cdot)$ for each $j$, we can choose smaller $r_j$ so that 
\begin{align*}
    \lim_{j\to\infty}\sum_{i=1}^3\Ll|\fint_{C_{j,\eta_j,\eps_j,r_j}}\mathsf{L}^i_{j,\eta_j,\eps_j}(t',\lf_j x') -\mathsf{L}^i_{j,\eta_j,\eps_j}(t_{j,\eta_j,\eps_j},\lf_j x_{j,\eta_j,\eps_j})\Rr|_i=0
\end{align*}
which along with~\eqref{e.limsupL^1} yields~\eqref{e.limsupL}. In the following, we turn to the terms on the right-hand sides in the next three steps.

\textbf{Step~2.} 
Setting $\tilde\bsigma_{j,\eta,\eps,r} = \frac{1}{\bsigma_{j,\eta,\eps,r}(\C_2^j)} \bsigma_{j,\eta,\eps,r}(\cdot\cap \C_2^j)$, we want to show
\begin{align}\label{e.tR-R}
    \lim_{j\to\infty}\sum_{i=1}^3\Ll|\int_{\cL^j}\tilde{\mathsf{R}}^i_j(\lf_j\cdot)\d \bsigma_{j,\eta_j,\eps_j,r_j} - \int_{\C_2^j}\mathsf{R}^i_j(\lf_j\cdot)\d \tilde\bsigma_{j,\eta_j,\eps_j,r_j}\Rr|_i=0,
\end{align}
where
\begin{gather*}
    \mathsf{R}^1_{j}({q'}) = \psi\Ll({q'}\Rr) - \la {q'},\Ll(\bnabla\psi({q'})\Rr)^\j\ra_{\cH},
    \\
    \mathsf{R}^2_{j}({q'}) = \bxi\Ll(\Ll(\nabla \psi\Ll({q'}\Rr)\Rr)^\j\Rr),\qquad \mathsf{R}^3_{j}({q'}) = \Ll(\nabla \psi\Ll({q'}\Rr)\Rr)^\j.
\end{gather*}
Since $\H$ is Lipschitz as in~\ref{i.d2} (below~\eqref{e.H=}), by the Lipschitzness of $\bnabla\psi$ in~\ref{i.A} along with Lemma~\ref{l.ext} (and Lemma~\ref{l.basics_(j)}~\eqref{i.isometric} and~\eqref{i.|iota^j|<|iota|}), we have
\begin{align}\label{e.tR<}
    \sum_{i=1}^3\Ll|\tilde{\mathsf{R}}^i_j({q'})\Rr|_i\leq C_j(1+|{q'}|_\cH),\quad\forall {q'}\in\cH,
\end{align}
for some constant $C_j>0$ depending on $j$ (due to Lemma~\ref{l.ext}).
By assumption~\ref{i.p_i} and Lemma~\ref{l.supp_cone} (with $\underline{x}_j$ substituted for $\underline x$), we have can make $\eta_j, \eps_j, r_j$ smaller so that
\begin{align}\label{e.supsigma}
    \bsigma_{j,\eta_j,\eps_j,r_j}\Ll(\cL^j\setminus\C_2^j\Rr) \leq j^{-1}\Ll(1\wedge C_j^{-1}\Rr),\quad\forall j\in\N.
\end{align}
Then, writing $\bsigma_{j,\cdots} = \bsigma_{j,\eta_j,\eps_j,r_j}$ and $\tilde\bsigma_{j,\cdots} = \tilde\bsigma_{j,\eta_j,\eps_j,r_j}$, we have
\begin{align*}
    &\Ll|\int_{\cL^j}\tilde{\mathsf{R}}^i_j(\lf_j\cdot)\d \bsigma_{j,\cdots} - \int_{\C_2^j}\tilde{\mathsf{R}}^i_j(\lf_j\cdot)\d \tilde\bsigma_{j,\cdots}\Rr|_i 
    \\
    &\leq \int_{\cL^j\setminus\C_2^j}\Ll|\tilde{\mathsf{R}}^i_j(\lf_j\cdot)\Rr|_i \d \bsigma_{j,\cdots} +  \Ll|\int_{\C_2^j}\tilde{\mathsf{R}}^i_j(\lf_j\cdot) \d \bsigma_{j,\cdots} - \int_{\C_2^j}\tilde{\mathsf{R}}^i_j(\lf_j\cdot) \d\tilde \bsigma_{j,\cdots}\Rr|,
    \\
    & \leq \Ll(\bsigma_{j,\cdots}(\cL^j\setminus\C_2^j)\Rr)^\frac{1}{2}\Ll(\int_{\C_2^j}\Ll|\tilde{\mathsf{R}}^i_j(\lf_j\cdot)\Rr|^2_i \d \bsigma_{j,\cdots}\Rr)^\frac{1}{2} + \Ll(\frac{1}{\bsigma_{j,\cdots}(\C_2^j)}-1\Rr) \int_{\C_2^j}\Ll|\tilde{\mathsf{R}}^i_j(\lf_j\cdot)\Rr|_i \d \bsigma_{j,\cdots}
\end{align*}
where we used the Cauchy-Schwarz inequality on the third line.
Using~\eqref{e.tR<}, the second moment bound in Lemma~\ref{l.2ndmomet}, and~\eqref{e.supsigma}, we can bound the above by
\begin{align*}
    C\Ll(\Ll(\bsigma_{j,\eta_j,\eps_j,r_j}(\cL^j\setminus\C_2^j)\Rr)^\frac{1}{2}C_j^\frac{1}{2} + \frac{\bsigma_{j,\eta_j,\eps_j,r_j}(\cL^j\setminus\C_2^j)}{1-\bsigma_{j,\eta_j,\eps_j,r_j}(\cL^j\setminus\C_2^j)} C_j\Rr) = o(1)
\end{align*}
for some $C$ independent of $j$, uniformly as $j\to\infty$. Therefore, we get
\begin{align}\label{e.tR-tR=0}
    \lim_{j\to\infty} \Ll|\int_{\cL^j}\tilde{\mathsf{R}}^i_j(\lf_j\cdot)\d \bsigma_{j,\eta_j,\eps_j,r_j} - \int_{\C_2^j}\tilde{\mathsf{R}}^i_j(\lf_j\cdot)\d \tilde\bsigma_{j,\eta_j,\eps_j,r_j}\Rr|_i  =0 .
\end{align}

For $y\in\C_2^j$, we have $\pj_j\lf_j y = y$ due to Lemma~\ref{l.basics_(j)}~\eqref{i.pj_lf=ID} and $\lf_j \nabla \psi^j(y) = (\bnabla\psi(\lf_j y))^\j$ due to Lemma~\ref{l.derivatives}~\eqref{i.char_nabla_j}. By this, $\tilde{\psi^j}=\psi^j$ on $\C_2^j$, and $\H((\bnabla\psi(\lf_j y))^\j)=\bxi((\bnabla\psi(\lf_j y))^\j)$\footnote{By~\ref{i.A}, we have $\bnabla\psi(\lf_j y) \in \mcl Q_\infty\subset \mcl Q_2$, which by~\eqref{e.kappa^(j)} and Lemma~\ref{l.proj_cones} implies $(\bnabla\psi(\lf_j y))^\j\in\mcl Q_2$. Assumption~\ref{i.A} also ensures $|\bnabla\psi(\lf_j y)|_{L^\infty}\leq1$. In view of~\eqref{e.kappa^(j)}, we also have $|(\bnabla\psi(\lf_j y))^\j|_{L^\infty}\leq1$. Hence, property~\ref{i.c1} of $\bar\xi$ (above~\eqref{e.H=}) gives $\bar \xi ((\bnabla\psi(\lf_j y))^\j) = \xi ((\bnabla\psi(\lf_j y))^\j)$. Then, property \ref{i.d1} of $\H$ (below~\eqref{e.H=}) yields $\H((\bnabla\psi(\lf_j y))^\j)=\bxi((\bnabla\psi(\lf_j y))^\j)$.}, we have
\begin{align*}\tilde{\mathsf{R}}^i_j(\lf_j\cdot) = \mathsf{R}^i_j(\lf_j\cdot)\quad\text{on $\C_2^j$,}\qquad \forall i\in\{1,2,3\},\ j\in\N.
\end{align*}
Hence, \eqref{e.tR-tR=0} becomes~\eqref{e.tR-R}.

\textbf{Step~3.}
Setting $\gamma_j = \tilde\bsigma_{j,\eta_j,\eps_j,r_j}\circ \lf_j^{-1}$, we want to show
\begin{align}\label{e.tRsigma=R_infty_gamma_j}
    \lim_{j\to\infty}\sum_{i=1}^3\Ll|\int_{\mcl Q^j_2}\mathsf{R}^i_j(\lf_j\cdot)\d \tilde\bsigma_{j,\eta_j,\eps_j,r_j} - \int_{\C_2}\mathsf{R}^i_\infty \d \gamma_j\Rr|_i=0,
\end{align}
where
\begin{align}\label{e.R_infty=}
    \mathsf{R}^1_\infty({q'}) =  \psi({q'}) -\la{q'}, \bnabla\psi({q'})\ra_{\cH}, \quad
    \mathsf{R}^2_\infty({q'}) =  \bxi\left(\bnabla\psi({q'})\right),\quad \mathsf{R}^3_\infty({q'}) = \bnabla\psi({q'}) 
\end{align}
correspond to the terms on the right-hand sides in~\eqref{e.p}.
Since $\lf_j^{-1}(\C_2 )=\pj_j(\C_2) =\C_2^j$ due to Lemma~\ref{l.basics_(j)}~\eqref{i.pj_lf=ID} and Lemma~\ref{l.proj_cones}, we have
\begin{align}\label{e.gamma_j(Q_2)=1}
    \gamma_j(\mcl Q_2) =1,\quad\forall j\in\N.
\end{align}
The definition of $\gamma_j$ along with~\eqref{e.tR-R} and~\eqref{e.gamma_j(Q_2)=1} gives
\begin{align}\label{e.tilde_R_j-R_j}
    \lim_{j\to\infty}\sum_{i=1}^3\Ll|\int_{\mcl Q^j_2}\mathsf{R}^i_j(\lf_j\cdot)\d \tilde\bsigma_{j,\eta_j,\eps_j,r_j} - \int_{\C_2}\mathsf{R}^i_j \d \gamma_j\Rr|_i=0.
\end{align}
Then, \eqref{e.tRsigma=R_infty_gamma_j} follows from \eqref{e.tilde_R_j-R_j} and the display below:
\begin{align}\label{e.R_j-R_infty}
    \lim_{j\to\infty}\sum_{i=1}^3\Ll|\int_{\C_2} \mathsf{R}^i_j \d \gamma_j  - \int_{\C_2}\mathsf{R}^i_\infty \d \gamma_j\Rr|_i=0.
\end{align}
Let us verify~\eqref{e.R_j-R_infty}.
Lemma~\ref{l.2ndmomet} together with~\eqref{e.supsigma} implies
\begin{align}\label{e.gamma_j_1st_mom}
    \sup_{j\in\N}\int_{\cH}|{q'}|^2_{\cH}\d\gamma_j({q'}) =\sup_{j\in\N}\int_{\cL^j}|y|^2_{\cL^j}\d \tilde\bsigma_{j,\eta_j,\eps_j,r_j}(y)<\infty.
\end{align}
By Lemma~\ref{l.L^1cvg} and~\ref{i.A}, there is some absolute constant $C$ such that, for every ${q'}\in\C_2$,
\begin{align*}\Ll|\Ll(\nabla \psi({q'})\Rr)^\j-\bnabla\psi({q'})\Rr|_\cH^2 \leq 2\Ll|\bnabla\psi({q'})\Rr|_{L^\infty}\Ll|\Ll(\nabla \psi({q'})\Rr)^\j-\bnabla\psi({q'})\Rr|_{L^1} \leq  C2^{-\frac{j}{2}}.
\end{align*}
Using this, the straightforward estimate (for $i=2$, we use the local Lipschitzness of $\xi$)
\begin{align*}
    \sum_{i=1}^3\Ll|\mathsf{R}^i_j - \mathsf{R}^i_\infty\Rr|_i\leq C(|{q'}|_\cH+1)\Ll|\Ll(\nabla \psi({q'})\Rr)^\j-\bnabla\psi({q'})\Rr|_\cH, 
\end{align*}
(for some other constant $C>0$) and~\eqref{e.gamma_j_1st_mom}, we can get~\eqref{e.R_j-R_infty}.

Lemma~\ref{l.weak_cvg} along with~\eqref{e.gamma_j_1st_mom} implies that, by passing to a subsequence, we can assume that $\gamma_j$ converges weakly to some $\gamma_\infty$. By~\eqref{e.gamma_j(Q_2)=1} and the Portmanteau theorem, we have
\begin{align}\label{e.gamma_infty(Q_2)=1}
    \gamma_\infty(\mcl Q_2) =1.
\end{align}

\textbf{Step~4.}
We want to show
\begin{align}\label{e.R_infty-R_infty}
    \lim_{j\to\infty}\Ll|\int_{\C_2} \mathsf{R}^i_\infty \d \gamma_j  - \int_{\C_2}\mathsf{R}^i_\infty \d \gamma_\infty\Rr|_i=0,\quad\forall i\in\{1,2,3\}.
\end{align}
For $M>0$ to be chosen, we set $\mathsf{R}^{i,M}_\infty = (\mathsf{R}^i_\infty \wedge M)\vee(-M)$ for $i\in\{1,2\}$. We split
\begin{align*}
    \Ll|\int_{\C_2} \mathsf{R}^i_\infty \d \gamma_j  - \int_{\C_2}\mathsf{R}^i_\infty \d \gamma_\infty\Rr|_i
    \leq \sum_{l\in\{j,\infty\}}  \Ll|\int_{\C_2} \mathsf{R}^i_\infty  - \mathsf{R}^{i,M}_\infty \d \gamma_l \Rr|_i
    +\Ll|\int_{\C_2} \mathsf{R}^{i,M}_\infty \d \gamma_j  - \int_{\C_2}\mathsf{R}^{i,M}_\infty \d \gamma_\infty\Rr|_i.
\end{align*}
By~\eqref{e.gamma_j_1st_mom}, we have $\int_{\cH}|{q'}|_{L^2}^2\d \gamma_\infty({q'})<\infty$ (see \cite[Theorem~25.11]{billingsley2012probability}).
Using this,~\eqref{e.gamma_j_1st_mom}, and $|\mathsf{R}^i_\infty|_i \leq C(1+|{q'}|_\cH)$ due to~\ref{i.A}, we have, for $l\in\{j,\infty\}$ and $i\in\{1,2\}$,
\begin{align*}
    \Ll|\int_{\C_2} \mathsf{R}^i_\infty -\mathsf{R}^{i,M}_\infty \d \gamma_l \Rr|_i = \rv{\Ll|\int_{\C_2} \Ll(\mathsf{R}^i_\infty -\mathsf{R}^{i,M}_\infty \Rr)\mathds{1}_{|\mathsf{R}^i_\infty|_i>M} \d \gamma_l \Rr|_i}
    \leq \int_{\C_2} \Ll|\mathsf{R}^i_\infty - \mathsf{R}^{i,M}_\infty\Rr|_i \mathds{1}_{|\mathsf{R}^i_\infty|_i>M} \d \gamma_l
    \\
    \leq 2\int_{\C_2}\Ll|\mathsf{R}^i_\infty\Rr| \mathds{1}_{|\mathsf{R}^i_\infty|_i>M} \d \gamma_l
    \leq 2M^{-1}\int_{\C_2} \Ll|\mathsf{R}^i_\infty \Rr|_i^2 \d \gamma_l 
    \leq CM^{-1}.
\end{align*}
By this and the previous display,
choosing $M$ sufficiently large and using the weak convergence of $\gamma_j$, we get~\eqref{e.R_infty-R_infty} for $i\in \{1,2\}$.

For $i=3$, we expand the square
\begin{align*}
    \Ll|\int_{\C_2} \mathsf{R}^3_\infty \d \gamma_j  - \int_{\C_2}\mathsf{R}^3_\infty \d \gamma_\infty\Rr|^2_\cH.
\end{align*}
From the assumption~\ref{i.A} on $\psi$, we can see that the function $({q'},q'')\mapsto \la \mathsf{R}^3_\infty({q'}), \mathsf{R}^3_\infty(q'')\ra_\cH$ is bounded and Lipschitz. It is straightforward that $\gamma_j\otimes \gamma_j$ and $\gamma_j\otimes \gamma_\infty$ converge weakly to $\gamma_\infty\otimes \gamma_\infty$. Therefore, we obtain \eqref{e.R_infty-R_infty} for $i=3$.

\textbf{Conclusion}.
Combining~\eqref{e.tR-R},~\eqref{e.tRsigma=R_infty_gamma_j} and~\eqref{e.R_infty-R_infty}, we get
\begin{align*}
    \lim_{j\to\infty}\sum_{i=1}^3\Ll|\int_{\cL^j}\tilde{\mathsf{R}}^i_j(\lf_j\cdot)\d \bsigma_{j,\eta_j,\eps_j,r_j} - \int_{\C_2}\mathsf{R}^i_\infty \d \gamma_\infty\Rr|_i=0.
\end{align*}
Using this, \eqref{e.L-R=0}, and~\eqref{e.limsupL}, we obtain
\begin{align*}
    \mathsf{L}^i_\infty = \int_{\C_2}\mathsf{R}^i_\infty \d \gamma_\infty,\quad i\in\{1,2,3\},
\end{align*}
for $\mathsf{L}^i_\infty$ and $\mathsf{R}^i_\infty$ defined in \eqref{e.L_infty=} and~\eqref{e.R_infty=},
which gives the desired results in~\eqref{e.p} for $\gamma_{t,{q}}=\gamma_\infty$. The support of $\gamma_\infty$ is verified in~\eqref{e.gamma_infty(Q_2)=1}.
\end{proof}

We are now ready to prove Theorems~\ref{t} and~\ref{t2}.

\begin{proof}[Proof of Theorem~\ref{t2}]
We first verify that the sequence in~\eqref{e.(Df_j...)} is precompact.
This follows from the uniform $L^\infty$-bounds on the derivatives in Lemma~\ref{l.bound_der_f_j}, the compactness in the weak-$L^2$ topology, and Lemma~\ref{l.cvg_L^r} (allowing us to upgrade the weak convergence to the strong one). 

Next, to show~\eqref{e.t2}, we want to apply Proposition~\ref{p.2nd_eq_diff_pt}. 
Let $(t_{j,\eta,\eps},x_{j,\eta,\eps})_{j\in\N,\eta>0,\eps>0}$ be given by Lemma~\ref{l.touch_approx}.
Let $\seq$ be the sequence along which the sequence in~\eqref{e.(Df_j...)} converges to $(a,p)$. 
Then, \ref{i.p_ii} and~\ref{i.p_iii} follow from Lemma~\ref{l.touch_approx} and the assumption in Theorem~\ref{t2} that $(t_j,\lf_jx_j)$ converges to $(t,q)$. To verify this, we also use \rv{the basic results} in Lemma~\ref{l.basics_(j)}~\eqref{i.isometric} and Lemma~\ref{l.derivatives}~\eqref{i.lift_gradient}.
To verify~\ref{i.p_i}, we use the assumption $x_j \in \itr(\mcl Q^j_2)$ to find $\underline{x}_j$ such that a small open neighborhood at $x_j$ is contained in $\underline{x}_j+\mcl Q^j_2$. Then, \ref{i.p_i} follows from this~\eqref{e.(t,x)->(t,mu)}. Now, we can apply Proposition~\ref{p.2nd_eq_diff_pt} to find some $\gamma_{t,q}$ so that~\eqref{e.p} and thus~\eqref{e.t2} hold.

Lastly, we verify~\eqref{e.t2_particular}.
For each $j\in\seq$, since $f_j$ is differentiable at $(t_j,x_j)$ and $f_j$ is the solution to $\HJ(\itr(\mcl Q^j_2),\bxi^j;\psi^j)$ as in~\eqref{e.f_j=}, we can verify using Definition~\ref{d.vis_sol} that
\begin{align*}
    \partial_t f_j(t_j,x_j) = \bxi^j\Ll(\nabla f_j(t_j,x_j)\Rr)\stackrel{\eqref{e.j-projection}}{=}\bxi\Ll(\lf_j\Ll(\nabla f_j(t_j,x_j)\Rr)\Rr).
\end{align*}
Using the convergence of the sequence in~\eqref{e.(Df_j...)} to $(a,p)$ and the uniform bounds in Lemma~\ref{l.bound_der_f_j}, we can get $a= \bxi (p)$. Now, \eqref{e.t2_particular} follows from this and~\eqref{e.t2}.
\end{proof}

\begin{proof}[Proof of Theorem~\ref{t}]
For each $j\in\N$, let $f_j$ be given as in~\eqref{e.f_j=} and we consider $(t,\pj_jq)$ with the projection $\pj_j$ given in~\eqref{e.pj,lj=}. Since $f_j$ is Lipschitz, it is differentiable almost everywhere by Rademacher's theorem. We can choose a differentiable point $(t_j,x_j) \in (0,\infty)\times \itr(\mcl Q^j_2)$ of $f_j$ to satisfy
\begin{align*}
    \Ll|(t_j,x_j) - (t,\pj_j q)\Rr|_{\R\times\cL^j}\leq j^{-1}.
\end{align*}
By~\eqref{e.kappa^(j)} and Lemma~\ref{l.basics_(j)}~\eqref{i.isometric} and~\eqref{i.cvg}, we can verify that $(t_j,\lf_jx_j)$ converges to $(t,q)$ in $\R\times L^2$. Then, we can apply Theorem~\ref{t2} to get that~\eqref{e.t2} holds for some $(a,p)\in\R\times L^2$. Inserting the expressions of $a$ and $p$ on the second line of~\eqref{e.t2} to the first line of~\eqref{e.t2}, we get~\eqref{e.t}.
\end{proof}

\smallskip

\noindent \textbf{Acknowledgment.} The author thanks Jean-Christophe Mourrat for help discussions.

\noindent \textbf{Funding.} The author is funded by the Simons Foundation. 

\noindent
\textbf{Data availability.}
No datasets were generated during this work.

\noindent
\textbf{Conflict of interests.}
The author has no conflicts of interest to declare.

\noindent
\textbf{Competing interests.}
The author has no competing interests to declare.

\small
\bibliographystyle{plain}
\newcommand{\noop}[1]{} \def\cprime{$'$}

\end{document}